\documentclass[10pt]{article}

\usepackage[utf8]{inputenc}
\usepackage[english]{babel}

\usepackage{amsmath}
\usepackage{enumerate}
\usepackage{amsthm,amssymb}
\usepackage{amscd}
\usepackage{hyperref}
\usepackage[a4paper,left=1in,top=1in,right=1in,
bottom=1in]{geometry}

 \newtheorem{low}{\textbullet \,
Lower limit assumption}
\newenvironment{lowerlimit}[1]
 {\renewcommand\thelow{#1}\low}
 
  \newtheorem{up}{\textbullet \,
Upper limit assumption}
\newenvironment{upperlimit}[1]
 {\renewcommand\theup{#1}\up}
 
   \newtheorem{reg}{\textbullet \,
Regularity assumption}
\newenvironment{regularity}[1]
 {\renewcommand\thereg{#1}\reg}
 
    \newtheorem{conf}{Confining sequence}
\newenvironment{confining}[1]
 {\renewcommand\theconf{#1}\conf}

    \newtheorem{sta}{Stable sequence}
\newenvironment{stable}[1]
 {\renewcommand\thesta{#1}\sta}

\newtheorem{theorem}{Theorem}
\newtheorem{corollary}[theorem]{Corollary}
\newtheorem{definition}[theorem]{Definition}
\newtheorem{proposition}[theorem]{Proposition}
\newtheorem{lemma}[theorem]{Lemma}
\newtheorem{remark}[theorem]{Remark}

\renewenvironment{abstract}
 {\par \noindent\textbf{Abstract.}\ \ignorespaces}
 {\par\vspace{5mm}}

\newenvironment{resume}
 {\par \noindent\textbf{Résumé.}\ \ignorespaces}
 {\par\vspace{3mm}}
 
\newenvironment{msc}
 {\par \noindent\textit{2010 MSC:}\ \ignorespaces}
 {\par\vspace{3mm}}
 
\newenvironment{keywords}
 {\par \noindent\textit{Keywords:}\ \ignorespaces}
 {\par\vspace{3mm}} 

\numberwithin{theorem}{section}

\numberwithin{equation}{section}

\usepackage{fancyhdr}
\title{A large deviation principle for
empirical measures
on Polish spaces: Application 
to singular Gibbs measures on manifolds}

\author{David García-Zelada
\\
\textit{\footnotesize
 CEREMADE, UMR CNRS 7534 Université 
Paris-Dauphine, PSL Research university,}
\\ 
\textit{\footnotesize Place du Maréchal
de Lattre de Tassigny 75016 Paris, France.}
\\
\textit{\footnotesize E-mail: garciazelada@ceremade.dauphine.fr}}

\begin{document}

\maketitle

\hrule
\vspace{7mm}

\begin{abstract}
\small
We prove a large deviation principle
for a sequence of point processes defined
by Gibbs probability measures 
on a Polish space. This is obtained as
a consequence of a more general
Laplace principle for
the non-normalized Gibbs measures.
We consider four main applications:
Conditional Gibbs measures on compact spaces, Coulomb
gases on compact Riemannian manifolds,
the usual Gibbs measures in the Euclidean space
and the zeros of Gaussian random polynomials.
Finally, we study
the generalization of Fekete points
and prove a
deterministic version of the Laplace principle
known as $\Gamma$-convergence.
The approach is partly inspired by
the works of Dupuis and co-authors.
It is remarkably natural and general
compared to the usual strategies
for singular Gibbs measures.
\end{abstract}

\begin{resume}
\small
On montre un principe de grandes déviations pour une suite de processus ponctuels définit par des mesures de probabilités de Gibbs dans un espace polonais. Il est obtenu comme conséquence d’un principe de Laplace pour des mesures de Gibbs non normalisées. On considère quatre applications: Des mesures de Gibbs conditionnées dans des espaces compacts, des gaz de Coulomb sur des variétés riemanniennes compactes, les mesures de Gibbs habituelles sur l’espace euclidien et les zéros des polynômes aléatoires gaussiens. Finalement, on étudie la généralisation des points Fekete et on prouve une version déterministe du principe de Laplace appelée $\Gamma$-convergence. Notre approche est partiellement inspirée par les travaux de Dupuis et ses coauteurs. C’est notablement naturelle et générale en comparaison avec les stratégies habituelles pour les mesures de Gibbs singulières.
\end{resume}

\begin{msc}
\small
60F10; 60K35; 82C22; 30C15
\end{msc}

\begin{keywords} 
\small
Gibbs measure; 
Coulomb gas; Empirical measure; Large deviation principle; 
Interacting particle system;
Singular potential;
Constant curvature;
Relative entropy; Random polynomials;
Fekete points
\end{keywords}

\vspace{3mm}

\hrule

\section{Introduction}

\label{section: introduction}

The present article 
is inspired by part of the work of 
Dupuis, Laschos and Ramanan on large deviations
for a sequence of point processes 
given by Gibbs measures associated to
very general singular two-body interactions
\cite{adupuis} but it differs
from it in that we take a general
sequence of interactions
that includes, for instance,
the interaction followed by
the zeros of random polynomials
as in 
\cite{zelditch}. We follow
the philosophy of 
Dupuis and Ellis 
\cite{bdupuis} about
the use of variational formulas
to make plausible and
sometimes easier to find a Laplace principle.
This philosophy
has already been used 
by Georgii in \cite{zgeorgii}
to treat a system of random fields on $\mathbb Z^d$
with interacting energies that converges
uniformly to
some limit functional.

We are interested in 
proving the Laplace principle
and the large deviation principle
for a very general sequence of energies
in a not necessarily compact space.
Part of our work has an overlap with
the article of Berman
\cite{berman2} and it was developed independently.
As in \cite{berman2} the interest of this result
is the generality of the sequence of energies:
they do not need to be made of a two-body
interaction potential but they may still be
very singular.
The key argument of the proof
is a well-understood application
of Jensen's inequality
together with a general Laplace principle
that has as its main ingredient
a subadditivity property of the entropy.
It is very simple compared
to the ad hoc methods 
used in the usual proofs of the large deviation
principles for Coulomb gases such as in 
\cite{hiaiunitary}, \cite{hiaigaussian}
\cite{chafai} and \cite{hardy}. 
In these methods, to prove a large deviation
lower bound, the authors usually
decompose the space
in small regions and this decomposition
may not be easy to achieve on a manifold and
not so natural to look for. We give a more precise
explanation of these methods
in Remark \ref{remark: usual}. 

Among the applications we can give
we are particularly interested
in explaining a simple case inspired by
\cite{bermanKEmetrics}. This is the
case of a Coulomb gas on a two-dimensional 
Riemannian manifold.
As a second application 
we study
a large deviation principle for a 
conditional Gibbs measure, i.e. we fix the position of some
of the particles and leave the rest of them random. 
The last applications we discuss are 
different proofs of already known results
such as the special one-dimensional log-gas of
\cite{ldpwigner} related
to the Gaussian ensembles, the
more general one-dimensional log-gas of
\cite[Section 2.6]{guionnet},
the special two-dimensional log-gas \cite{hiaigaussian} 
related to the Ginibre ensemble of random matrices 
and its generalization
to an $n$-dimensional Coulomb gas
in \cite{chafai} and \cite{adupuis}, the note in 
 \cite{hardy} about two-dimensional log-gases
with a weakly confining potential and
the Gaussian random polynomials of
\cite{zelditch} and \cite{raphael}.

We now explain the contents of each section.
The rest of Section \ref{section: introduction} will be dedicated to 
the main definitions and assumptions we will 
need to state our results.
Section \ref{section: kbody} is about
the usual mean-field case,
the $k$-body interaction. We give sufficient
conditions to be able to apply our result
which will become important
when we treat the Euclidean space case.
In Section \ref{section: proof theorems} we 
begin by giving an idea of the proofs which
includes mainly 
a key variational formula.
Then we give the proofs
of the main theorem and of its corollary
and we finish the section by giving some
remarks about the usual proofs
we may find in the literature.
We discuss four particular
examples in Section \ref{section: applications}.
More precisely, the conditional Gibbs measure,
the Coulomb gas on a Riemannian manifold,
a new 
way to obtain already known results in the Euclidean space about Coulomb gases and the assertion that
the zeros of a Gaussian random polynomials may
be treated by our main theorem.
We conclude our article with Section
\ref{section: fekete} discussing a deterministic
case which falls under the topic
of Fekete points and which we consider as the natural
deterministic analogue of the Laplace principle.

\subsection{Model}

Let $M$ be a Polish space,
i.e. a separable topological space
metrizable by a complete metric.
Endow it with the Borel $\sigma$-algebra associated
to this topology, i.e. the least $\sigma$-algebra
that contains the topology.
Denote by $\mathcal P(M)$
the space of probability measures in $M$ and
endow it with the
smallest topology such that
$\mu \mapsto \int_M f d\mu$ is continuous for every 
bounded continuous function $f: M \to \mathbb R$.
With this topology, $\mathcal P(M)$ is also a Polish space
(see \cite[Section 2.4]{borkar}).
This is called the weak topology.
Suppose we have a sequence
$\{W_n \}_{n \in \mathbb N}$
of symmetric measurable functions 
	$$W_n: M^n \to (- \infty, \infty]$$
and a sequence of non-negative numbers 
$\{\beta_n\}_{n \in \mathbb N}$ that converges
to some $\beta \in (0, \infty]$.	
Fix a probability measure $\pi \in \mathcal P(M)$.
We shall be
interested in the
asymptotic behavior of the
 {\bf Gibbs measures} $\gamma_n$ defined by
	\begin{equation}	\label{eq: gibbs}
	d   {\gamma}_n =
	e^{-n \beta_n W_n} d  \pi^{\otimes_n}.
	\end{equation}
Define $\tilde W_n:  \mathcal P(M)  \to 
	(- \infty,  \infty]$ by
	\begin{equation} \label{eq: extension}
	\tilde W_n(\mu) = 
	\begin{cases}
	W_n(x_1,...,x_n) & \mbox{ if } \mu \mbox{ is atomic 
	with } \mu = \frac{1}{n}\sum_{i=1}^n \delta_{x_i}\\
	\infty & \mbox{ otherwise.}
	\end{cases}
	\end{equation}	
	
\begin{stable}{(S)}
\label{stablesequence}
We shall say that the sequence
$\{W_n \}_{n \in \mathbb N}$ is a 
{\bf stable sequence} if it is
uniformly bounded from below, i.e. if
there exists $C \in \mathbb R$ such that
	$$ W_n \geq C \mbox{ for all } n \in \mathbb N. $$
\end{stable}

\begin{confining}{(C)} 
\label{confiningsequence}
We shall
say that $\{W_n\}_{n \in \mathbb N}$
is a {\bf confining sequence}
if the following is true. Let
 $\{n_j\}_{j \in \mathbb N}$ be any
increasing sequence of natural numbers
and let $\{\mu_j\}_{j \in \mathbb N}$ be any sequence
of probability measures on $M$. If there exists
a real constant $A$ such that
	$$\tilde W_{n_j}(\mu_j) \leq A$$
for every $j  \in \mathbb N$, where $\tilde W_n$
is defined in (\ref{eq: extension}), then 
$\{\mu_j\}_{j \in \mathbb N}$ is relatively
compact in $\mathcal P(M)$.
\end{confining}
\ \\
In order to study the behavior as $n \to \infty$
of $\gamma_n$ we shall need a
measurable function
	$$W : \mathcal P(M) \to (- \infty, \infty].$$

\begin{definition}[Macroscopic limit]
Suppose that
$\{W_n \}_{n \in \mathbb N}$
is a {\it stable sequence} \ref{stablesequence}.
We say that 
a measurable function
$W : \mathcal P(M) \to (- \infty, \infty]$ is
the {\bf positive temperature macroscopic limit} of 
the sequence $\{W_n \}_{n \in \mathbb N}$ if
the following two conditions are satisfied.

\begin{lowerlimit}{(A1)}
\label{lowerlimitassumption}
	For every sequence $\{\mu_n\}_{n\in \mathbb N}$
	of probability measures on $M$ that
	 converges to some
	probability measure $\mu$ we have 
	$$
		\liminf_{n \to \infty} \tilde W_n(\mu_n) \geq W(\mu)
	$$
where $\tilde W_n$
is defined in (\ref{eq: extension}).	
\end{lowerlimit}

\begin{upperlimit}{(A2)}	
\label{upperlimitassumption}
	For each $\mu \in \mathcal P(M)$ we have that
	$$
	\limsup_{n \to \infty} \mathbb E_{\mu^{\otimes_n}}
		 [W_n] \leq W(\mu).
	$$
\end{upperlimit}		

We say that $W$ is the 
 {\bf zero temperature macroscopic limit} of 
the sequence $\{W_n \}_{n \in \mathbb N}$ if
instead
 the {\it lower limit assumption}
 \ref{lowerlimitassumption} and
the following condition are satisfied.

\begin{regularity}{(A2')} 
\label{regularityassumption}
Define the set of `nice' probability measures
	\begin{equation} \label{eq: nice}
	\mathcal N = \left\{ 
	\mu \in \mathcal P(M) :\ D(\mu \| \pi) < \infty
	 \mbox{ and } \limsup_{n \to \infty} \mathbb
	 E_{\mu^{\otimes_n}}
		 [W_n] \leq W(\mu)
		 \right\} .
	\end{equation}
For every $\mu \in \mathcal P(M)$ such
that $W(\mu) < \infty$ we can find a sequence
of probability measures
$\{\mu_n \}_{n \in \mathbb N}$ in $\mathcal N$
such that 
$\mu_n \to \mu$
and $\limsup_{n \to \infty} W(\mu_n) \leq W(\mu)$.
\end{regularity}

\end{definition}
Now we are ready to state the Laplace
principles and the large deviation principles.

\subsection{Main results}

Let
$i_n : M^n \to \mathcal P(M)	$
be the application defined by	
	\begin{equation} \label{eq: inclusion}
	i_n
	(x_1,...,x_n) \mapsto \frac{1}{n}\sum_{i=1}^n
	 \delta_{x_i},
	 \end{equation}
the usual continuous `inclusion' of $M^n$ in $\mathcal P(M)$.
Define the {\bf free energy} with parameter $\beta$
as
	\begin{equation}\label{eq: free}
	F = W + \frac{1}{\beta}D( \cdot \| \pi),
	\end{equation}
(we suppose $0 \times \infty = 0$)
where $D(\mu \| \nu)$ denotes the 
relative entropy of $\mu$ with
respect to $\nu$, also known as the 
Kullback–Leibler divergence i.e.
	\begin{equation}	\label{eq: entropy}
	D(\mu \| \nu) = \int_M \frac{d \mu}{d \nu}
	 \log \left(\frac{d \mu}{d \nu} \right) 
	d\nu	
	\end{equation}
if $\mu$ is absolutely continuous with respect to $\nu$
 and $D(\mu \| \nu) = \infty$ otherwise.

\begin{theorem}[Laplace principle]
\label{theorem: laplace}
Let $\{W_n \}_{n \in \mathbb N}$ be a stable
sequence \ref{stablesequence} and 
$W: \mathcal P(M) \to 	(- \infty,  \infty]$
 a measurable function.
Take a sequence 
of positive numbers $\{\beta_n\}_{n \in \mathbb N}$
that converges to some $\beta \in (0, \infty]$.

If $\beta < \infty$ suppose
that $W$ is the positive temperature macroscopic
limit of $\{W_n\}_{n \in \mathbb N}$.

If $\beta = \infty$ suppose 
that $W$ is the zero temperature macroscopic
limit of $\{W_n\}_{n \in \mathbb N}$ and
suppose that $\{W_n \}_{n \in \mathbb N}$  is a
{\it confining sequence} \ref{confiningsequence}.

Define the Gibbs measures $\gamma_n$ by 
(\ref{eq: gibbs})
and the free energy $F$ by (\ref{eq: free}).
Then, the following Laplace's principle is satisfied.

For every bounded continuous
function $f: \mathcal P(M) \to \mathbb R$
	$$\frac{1}{n \beta_n} \log \,  \int_{M^n} 
	 e^{- n \beta_n f \circ i_n} d \gamma_n 
	 \,
     \xrightarrow[\: n \to \infty \:]{}
     \,
	- \inf_{\mu \in \mathcal P(M) } \{ f \left( \mu \right) +
	 F \left( \mu \right) \}.$$	 
	 
\end{theorem}
\
\\ 
This Laplace principle implies the following
large deviation principle.

\begin{corollary}
[Large deviation principle]
\label{cor: deviations}

Suppose the same conditions as in Theorem
 \ref{theorem: laplace}.
Define $Z_n = \gamma_n (M^n)$.
Suppose $Z_n > 0$ for every $n$ and notice that,
as $W_n$ is bounded from below, $Z_n < \infty$. 
Take the sequence of probability measures
$\{\mathbb P_n\}_{n \in \mathbb N}$ defined by
\begin{equation} 
\label{eq: gibbs probability finite beta}
	d \mathbb P_n = \frac{1}{Z_n} d \gamma_n.
\end{equation}	

For each $n \in \mathbb N$, let $i_n(\mathbb P_n)$
be the pushforward measure of $\mathbb P_n$
by $i_n$.
Then the sequence 
$\{i_n(\mathbb P_n)\}_{n \in \mathbb N}$ 
satisfies a large deviation
principle with speed $n \beta_n$ and with rate function
	$$I = F - \inf F,$$
i.e.
for every open set $A \subset \mathcal P(M)$
we have
	$$\liminf_{n \to \infty} \frac{1}{n \beta_n}
	\log \mathbb P_n(i_n^{-1}(A)) \geq - 
	\inf_{\mu \in A} I(\mu)$$	
and for every closed set $C \subset \mathcal P(M)$
we have
	$$\limsup_{n \to \infty} \frac{1}{n \beta_n}
	\log \mathbb P_n(i_n^{-1}(C)) \leq - 
	\inf_{\mu \in C} I(\mu).$$	
\end{corollary}	 
\ \\
In the next section, Section \ref{section: kbody},
 we shall study the
usual case of $k$-body interaction.
Section \ref{section: applications} 
will be about
some more specific examples, 
such as
the conditional Gibbs measure, 
the Coulomb gas on a compact Riemannian manifold,
the usual Gibbs measures on a 
noncompact space such as the Euclidean space
and the Gaussian random polynomials.

\section{Example of a stable sequence: k-body interaction}

\label{section: kbody}

We will give 
the most basic non-trivial example of a
{\it stable sequence} \ref{stablesequence}. 
Take
an integer $k > 0$ and
a symmetric lower semicontinuous function
bounded from below
$G : M^k \to ( - \infty, \infty] .$
Define the symmetric
measurable functions
$W_n : M^n \to ( - \infty, \infty]$ by
	$$W_n(x_1,...,x_n) =
	\frac{1}{n^k}
	\sum_{\stackrel{\{i_1,...,i_k\}  
	\subset \{1,...,n\}}
	{ \# \{i_1,...,i_k\} = k}}
	 G(x_{i_1},...,x_{i_k}) .$$
and $W: \mathcal P(M) \to ( - \infty, \infty]$ by
	$$W(\mu) =  \frac{1}{k!}\int_{M^k} G(x_1,...,x_k) 
	d\mu(x_1) ... d\mu(x_k).$$
\begin{proposition}[Stability,
lower and upper limit assumption, 
(A1) and (A2)]
\label{proposition: k-body}
$\{W_n\}_{n \in \mathbb N}$ is a 
stable sequence \ref{stablesequence},
$W$ is lower semicontinuous and
the pair $(\{W_n\}_{n \in \mathbb N}, W)$ satisfies
the {\it lower and upper limit assumption}, 
\ref{lowerlimitassumption} and
\ref{upperlimitassumption}.

\begin{proof}

To see that $\{W_n\}_{n \in \mathbb N}$ 
is a {\it stable sequence}
\ref{stablesequence}
we notice that if $C \leq G$ then 
${n \choose k}\frac{1}{ n^k} C \leq W_n$.
The lower semicontinuity of $W$ is 
a consequence of the lower semicontinuity of
 $G$ and the fact
 that it is bounded from below. Now,
let us prove that $(\{W_n\}_{n \in \mathbb N}, W)$ satisfies
the {\it lower and upper limit assumption},
\ref{lowerlimitassumption} and
\ref{upperlimitassumption}.

$\bullet$ {\bf Lower limit assumption 
\ref{lowerlimitassumption}.} 
Let $\mu \in \mathcal P(M)$. Take
 $N > 0$ and define $G_N = G \wedge N$. 
 We will prove that
\begin{equation} \label{eq: kenergy inequality}
	\tilde W_n(\mu) + 
	\frac{N}{k!\,  n^k} 
	 	\left(n^k
	  - \frac{n!}{(n - k)!} \right)	
	\geq
	 \frac{1 }{k!}\int_{M^k} G_N(x_1,...,x_k) 
	d\mu(x_1) ... d\mu(x_k)
\end{equation}
where $\tilde W_n$ is the extension defined
in (\ref{eq: extension}).	
If $\tilde W_n(\mu) = \infty$ there is nothing to prove.
If $\tilde W_n(\mu) < \infty$ then
$\mu=\frac{1}{n}\sum_{i=1}^n \delta_{x_i}$
for some $(x_1,...,x_n) \in M^n$. We have

	$$\frac{1}{n^k}
	\sum_{\stackrel{\{i_1,...,i_k\}  
	\subset \{1,...,n\}}
	{ \# \{i_1,...,i_k\} = k}}
	 G_N(x_{i_1},...,x_{i_k}) 
	 + \frac{N}{k!\,  n^k} 
	 	\left(n^k
	  - \frac{n!}{(n - k)!} \right)
	\geq \frac{1 }{k!}\int_{M^k} G_N(x_1,...,x_k) 
	d\mu(x_1) ... d\mu(x_k),$$
which due to the fact that $G \geq G_N$ implies the
inequality (\ref{eq: kenergy inequality}).
	
Let $\mu_n \to \mu \in \mathcal P(M)$. 
Then, using the inequality 
(\ref{eq: kenergy inequality}) and taking
the lower limit
we get
\begin{align*}
	\liminf_{n \to \infty} \tilde W_n(\mu_n) &\geq
	   \frac{1 }{k!}\int_{M^k} G_N(x_1,...,x_k) 
	d\mu(x_1) ... d\mu(x_k) ,
\end{align*}		
where we have used that $G_N$ is lower semicontinuous
and bounded from below.
Finally, as $G$ is bounded from below we can take
$N$ to infinity and use
the monotone convergence theorem to get
	$$\liminf_{n \to \infty} \tilde W_n(\mu_n) 
	 \geq
	\frac{1 }{k!}\int_{M^k} G(x_1,...,x_k) 
	d\mu(x_1) ... d\mu(x_k) .$$

$\bullet$ {\bf Upper limit assumption
\ref{upperlimitassumption}.}
For this it is enough to take $\mu \in \mathcal P(M)$
and notice that
	$$ \mathbb E_{\mu^{\otimes_n}}  [W_n]  =
		\frac{1}{n^k}{n \choose k}\int_{M^k} G(x_1,...,x_k) 
	d\mu(x_1) ... d\mu(x_k). $$

\end{proof}

\end{proposition}
\ \\
Now we give a sufficient condition
for a $k$-body interaction
to be a {\it confining sequence}
\ref{confiningsequence}.

\begin{proposition}[$k$-body interaction and confining
assumption]
\label{proposition: kbodyconfining}
Suppose $G(x_1,...,x_k)$ tends to infinity
when $x_i \to \infty$ for all $i \in \{1,...,k\}$,
i.e.  suppose that for every $C \in \mathbb R$ there exists 
a compact set $K$ such that 
$G|_{K^c \times ... \times K^c} \geq C$. Then
 $\{W_n\}_{n \in \mathbb N}$
is a {\it confining sequence} \ref{confiningsequence}.
\begin{proof}
Without loss of generality 
we can suppose $G$ positive. 
Remember the definition of $\tilde W_n$
in (\ref{eq: extension}). All we need
is the following result.

\begin{lemma}[Bound on the number of particles
outside a compact set]

\label{boundconfining}

Suppose that $G$ is positive.
Take $n \in \mathbb N$, 
$A \in \mathbb R$ and $\mu \in \mathcal P(M)$ 
that satisfies
	$$\tilde W_n(\mu) \leq A.$$
If $K$ is a compact set such that
$G|_{K^c \times ... \times K^c} \geq C$ 
with $C>0$, then
	$$\mu(K^c) \leq
	\left(\frac{A}{C} \, k! \right)^{1/k} + 
	\frac{k}{n}.$$

\begin{proof}

We first
notice that
 $\mu = \frac{1}{n}\sum_{i=1}^n \delta_{x_i}$
 for some $(x_1,...,x_n) \in M^n$.
By the hypotheses we can see that
\begin{equation} \label{eq: boundonkcombinations}
	 \frac{1}{n^k}
	[\mbox{number of $k$-combinations outside $K$}]
	 \, C
	\leq 
	\frac{1}{n^k}
	\sum_{\stackrel{\{i_1,...,i_k\}  \subset \{1,...,n\}}
	{ \# \{i_1,...,i_k\} = k}}
	 G(x_{i_1},...,x_{i_k}) \leq A.
\end{equation}
where, more precisely,
$[\mbox{number of $k$-combinations outside $K$}]$ 
denotes
the cardinal of the following set,
$\{S  \subset \{1,...,n\}: \
\# S = k \mbox{ and }
\forall i \in S, \, x_{i} \notin K  \}$.
But, if $m$ denotes the
number of points among $x_1,...,x_n$ outside $K$
and if $k \leq m$, we have
	$$\frac{(m-k)^k }{k!} \leq
	\frac{m!}{(m-k)!\, k!}
	=
	[\mbox{number of $k$-combinations outside $K$}]$$
which, along with the inequality 
(\ref{eq: boundonkcombinations}),
implies
	$$\frac{m}{n} \leq 
		\left(\frac{A}{C} \, k! \right)^{1/k}
		 + \frac{k}{n}.$$
As $\mu = \frac{1}{n}\sum_{i=1}^n \delta_{x_i}$ then
	$\mu(K^c) = \frac{m}{n}$ which concludes the proof. 
		
\end{proof}

\end{lemma}

Then we can conclude
using Prokhorov's theorem and the fact
that every single probability measure is tight.

\end{proof}
\end{proposition}
\ \\
Finally we notice that in the 
{\it regularity assumption} 
\ref{regularityassumption}
we can replace finite entropy by
absolute continuity with respect to $\pi$.

\begin{proposition}
[$k$-body interaction and regularity assumption]
\label{proposition: kbodyregularity}

Let
\begin{align*}
	\mathcal N_1 &=
	\left\{\mu \in \mathcal P(M) :\
	D(\mu \| \pi) < \infty
		 \right\} \ \mbox{ and} \\
	\mathcal N_2 & =
	\left\{\mu \in \mathcal P(M) :\
	\mu 
	\mbox{ is absolutely continuous with respect to }
	 \pi
	 \right\}.
\end{align*}	 
Suppose that for
every $\mu$ with $W(\mu) < \infty$, there exists
a sequence $\{\mu_n \}_{n\in \mathbb N}$
in $\mathcal N_2$ such that
$\mu_n \to \mu$ and $W(\mu_n) \to W(\mu)$
then the same is true if we replace $\mathcal N_2$ by
$\mathcal N_1$. 

\begin{proof}
It is enough to prove that 
for every $\mu \in \mathcal N_2$ 
there exists a sequence
$\{\mu_n\}_{n \in \mathbb N}$
in $\mathcal N_1$ such that $\mu_n \to \mu$
and $W(\mu_n) \to W(\mu)$. 
Let $\rho$ be the density of $\mu$
with respect to $\pi$, i.e. $d\mu = \rho \, d\pi$.
For each $n > 0$ define $\mu_n \in \mathcal N_1$ by 
$d \mu_n = \frac{\rho \wedge n\,  d \pi}
{ \int_M \rho \wedge n \,  d \pi}$.
Then, by the monotone
convergence theorem we can see that
$\mu_n \to \mu$. And, again, by the monotone convergence
 theorem, by supposing $G \geq 0$, we can see
 that $W(\mu_n) \to W(\mu)$.

\end{proof}
\end{proposition}

\section{Proof of the theorem} 

\label{section: proof theorems}

This section is dedicated to the proof of the main
theorem, i.e. 
Theorem \ref{theorem: laplace}.
We start giving a sketch of the proof.

\subsection{Idea of the proof}
\label{idea}

We shall use the 
following very known result that
tells us the Legendre transform of $D(\cdot \| \mu)$,
defined in (\ref{eq: entropy}).
See \cite[Proposition 4.5.1]{bdupuis} for a proof.
\begin{lemma}[Legendre transform of the entropy]
\label{lemma: legendre}
	Let $E$ be a Polish probability space,
	$\mu$ a probability measure on $E$
	 and $g:E \to ( - \infty, \infty]$
	 a measurable function bounded from below.
	 Then
	$$\log \mathbb E_\mu \left[e^{-g} \right] = 
	- \inf_{\tau \in \mathcal P(E)} 
	\left\{ \mathbb E_\tau \left[ g \right] + 
	D (\tau \| \mu)\right\}.$$
\end{lemma}
\ \\
Remember the definition
of $\gamma_n$ in (\ref{eq: gibbs}) and $F$ in 
(\ref{eq: free}).
  With the help of Lemma \ref{lemma: legendre} 
  we can write
\begin{align*}
\frac{1}{n\beta_n} \log \,  \int_{M^n} 
	 e^{- n\beta_n f \circ i_n} d \gamma_n 
&= \frac{1}{n\beta_n} \log \,  \mathbb E_{\pi^{\otimes_n} } 
	\left[ e^{- n\beta_n \left( f \circ \, i_n +  W_n \right)}\right]
		\\
	&=		
	- \inf_{\tau \in \mathcal P(M^n)}
	 \left\{ \mathbb E_{i_n(\tau)} \left[  f \right] + 
	 \mathbb E_\tau 
	  \left[ W_n  \right] + \frac{1}{n\beta_n}D (\tau \| \pi^{\otimes_n})\right\}.
\end{align*}
where $i_n(\tau)$ denotes the pushforward 
measure of
$\tau$ by $i_n$.
So, we need to prove that
\begin{equation} \label{eq: limit to prove}
  \inf_{\tau \in \mathcal P(M^n)}
	 \left\{ \mathbb E_{i_n(\tau)} \left[  f \right] + 
	 \mathbb E_\tau 
	  \left[ W_n  \right] + \frac{1}{n\beta_n}D (\tau \| \pi^{\otimes_n})\right\}
	 \xrightarrow[\: n \to \infty \:]{}
 	\inf_{\mu \in \mathcal P(M) } \{ f \left( \mu \right) +
	 F \left( \mu \right) \}.
\end{equation}	 
\subsection{Proof of 
Theorem \ref{theorem: laplace}:
Case of finite 
$\beta$ }	

In this subsection we shall prove
the Laplace principle Theorem \ref{theorem: laplace}
 and the
large deviation principle 
 Corollary \ref{cor: deviations}
for the case of finite $\beta$.

To prove this we need the following properties
of the entropy. The first
one is analogous to the  
{\it lower limit assumption} \ref{lowerlimitassumption}.

\begin{lemma}[Lower limit property of the entropy]
\label{lemma: entropy lower limit}
Let $\{n_j\}_{j \in \mathbb N}$
 be an increasing sequence in $\mathbb N$.  
 For each $j \in \mathbb N$ 
 take $\tau_j \in \mathcal P(M^{n_j})$.
 If $i_{n_j}(\tau_j) 
 \to \zeta \in \mathcal P( \mathcal P(M) )$,
 then
 $$\mathbb E_{\zeta} \left[ D \left(\cdot | \pi \right )
	 \right]  \leq 
	\liminf_{j \to \infty}  
	\frac{1}{n_j} D (\tau_j \| \pi^{\otimes_{n_j}}) .$$
\begin{proof}	
The idea of the proof is presented
in \cite{bdupuis}. We can also see
\cite{david}. It can be
seen as equivalent to the 
large deviation upper bound of Sanov's theorem
thanks to \cite[Theorem 3.5]{mauro}.
\end{proof}	
\end{lemma}
\ \\
And the second one is analogous to the 
notion of {\it confining sequence}
 \ref{confiningsequence}.

\begin{lemma}[Confining property of the entropy]
\label{lemma: entropy confining}
Let $\{n_j\}_{j \in \mathbb N}$
 be an increasing sequence in $\mathbb N$.  
 For each $j \in \mathbb N$ 
 take $\tau_j \in \mathcal P(M^{n_j})$. If
 there exists a real constant $C$ such that
 $$\frac{1}{n_j}D(\tau_j \| \pi^{\otimes_{n_j}} )
  \leq C$$
 for every $j \in \mathbb N$, then the sequence 
 $\{i_{n_j}(\tau_j)\}_{j \in \mathbb N}$ is tight.
\begin{proof} 
The idea of the proof is presented
in \cite{bdupuis}. We can also see
\cite{david}. It can be
seen as equivalent to the 
exponential tightness in Sanov's theorem
thanks to \cite[Theorem 3.3]{mauro}.
\end{proof}
\end{lemma}

Without loss of generality,
we can suppose $\beta_n =1$ 
for every $n$ by redefinition of $W_n$
and $W$. Then
the {\it Gibbs measure} (\ref{eq: gibbs}) and
the {\it free energy} (\ref{eq: free}) are
	$$d \gamma_n =
	e^{-n  W_n} d\pi^{\otimes_n} 
	\quad \text{and} \quad
	F = W + D( \cdot \| \pi).$$
As explained in Subsection \ref{idea} we need to prove
(\ref{eq: limit to prove}) which in this case is
 $$ \inf_{\tau \in \mathcal P(M^n)}
	 \left\{ \mathbb E_{i_n(\tau)} \left[  f \right] + 
	 \mathbb E_\tau 
	  \left[ W_n  \right] + \frac{1}{n}D (\tau \| \pi^{\otimes_n})\right\}
	 \xrightarrow[\: n \to \infty \:]{}
 	\inf_{\mu \in \mathcal P(M) } \{ f \left( \mu \right) +
	 F \left( \mu \right) \}.$$

\begin{proof}[Proof of Theorem 
\ref{theorem: laplace}: Case of finite $\beta$]

First, we will prove the
{\bf lower limit bound}
\begin{equation} \label{eq: lower limit bound}
\liminf_{n \to \infty} 
 \inf_{\tau \in \mathcal P(M^n)}
	 \left\{ \mathbb E_{i_n(\tau)} \left[  f \right] + 
	 \mathbb E_\tau 
	  \left[ W_n  \right] + \frac{1}{n}D (\tau \| \pi^{\otimes_n})\right\} \\
	  \geq 	  \inf_{\mu \in \mathcal P(M) } \{ f \left( \mu \right) + F \left( \mu \right)\}.
\end{equation}
This is equivalent to say that for every 
increasing sequence 
of natural numbers $\{n_j\}_{j \in \mathbb N}$
if we choose, for each $j \in \mathbb N$, a
probability measure
$\tau_j \in \mathcal P(M^{n_j})$ we have

\begin{equation} \label{eq: liminf part}
	\lim_{j \to \infty}\left\{	
	\mathbb E_{i_{n_j}(\tau_j)} \left[  f \right] + 
	 \mathbb E_{\tau_j}
	  \left[ W_{n_j}  \right] + \frac{1}{n_j} 
	  D (\tau_j \| \pi^{\otimes_{n_j}})
	  \right\}
	\geq
	\inf_{\mu \in \mathcal P(M) } 
	\{ f \left( \mu \right) + 
	F \left( \mu \right)\}. 
\end{equation}
where we can suppose that the limit exists
and that it is finite and, in particular,
the sequence is bounded from above.

Using that $\{W_n\}_{n \in \mathbb N}$ 
is a {\it stable sequence} \ref{stablesequence},
we get that 
	$\frac{1}{n_j} 
	  D (\tau_j \| \pi^{\otimes_{n_j}})$
is uniformly bounded from above.
By the {\it 
confining property of the entropy},
Lemma \ref{lemma: entropy confining}, we get that
$i_{n_j}(\tau_j)$ is tight. By taking a subsequence 
using Prokhorov's theorem,
we shall 
assume it converges to some 
$\zeta \in \mathcal P( \mathcal P(M) )$.
 Then, by the {\it 
lower limit property of the entropy}, 
Lemma \ref{lemma: entropy lower limit},
we get
 $$\mathbb E_{\zeta} 
 \left[ D \left(\cdot | \pi \right )
	 \right]  \leq 
	\liminf_{j \to \infty}  
	\frac{1}{n_j} D (\tau_{j} \| \pi^{\otimes_{n_j}}).
	$$
As $\tilde W_n$ is measurable 
for every $n$ (see \cite[Proposition 7.6]{david}
for a proof)
and the sequence 
$\{\tilde W_n\}_{n \in \mathbb N}$
is uniformly bounded from below we 
may use the {\it lower limit assumption} 
\ref{lowerlimitassumption}
to get (see \cite[Proposition 3.2]{mauro})
	$$\mathbb E_\zeta [W] 
	\leq
	   \liminf_{j \to \infty} \mathbb E_{\tau_{j}}
	    \left[	 W_{n_j} \right].$$
Then, by taking the lower limit 
when $j$ tends to infinity 
in (\ref{eq: liminf part}),
we obtain
	$$\lim_{j \to \infty} \left\{
	\mathbb E_{i_{n_j}(\tau_j)} \left[  f \right] + 
	 \mathbb E_{\tau_j}
	  \left[ W_{n_j}  \right] + \frac{1}{n_j} 
	  D (\tau_j \| \pi^{\otimes_{n_j}})
	  \right\}
	  \geq 
	  \mathbb E_{\zeta} \left[ f + 
	W +  D (\cdot \| \pi) \right]
	\geq
		\inf_{\mu \in \mathcal P(M) } 
	\{ f \left( \mu \right) + 
	F \left( \mu \right)\}. 
	  $$

\noindent
Now let us prove the {\bf upper limit bound}
\begin{equation} \label{eq: upper limit bound}
	 \limsup_{n \to \infty}  \inf_{\tau \in \mathcal 
	 P(M^n)}
	 \left\{ \mathbb E_{i_n(\tau)} \left[ f \right] + 
	 	\mathbb E_{\tau} \left[ W_n \right] +
	 \frac{1}{n} D (\tau \| \pi^{\otimes_n})\right\}
	\leq 	 \inf_{\mu \in \mathcal P(M) } 
	\{ f \left( \mu \right) + F \left( \mu \right)\}. 
\end{equation}
We need to prove that for every probability 
measure
$\mu \in \mathcal P(M)$
	$$ \limsup_{n \to \infty} 
	\inf_{\tau \in \mathcal P(M^n)}
	 \left\{ \mathbb E_{i_n(\tau)} \left[ f \right] + 
	 	\mathbb E_{\tau} \left[ W_n \right]+
	 \frac{1}{n} D (\tau \| \pi^{\otimes_n})\right\}
	 \leq 	  
	f \left( \mu \right) + F \left( \mu \right). $$
It is enough to find a sequence 
$\tau_n \in \mathcal P(M^n)$ such that
	$$\limsup_{n \to \infty} 
	\left\{ \mathbb E_{i_n(\tau_n)} \left[ f \right] + 
	\mathbb E_{\tau_n} \left[ W_n \right]+
	\frac{1}{n}D (\tau_n \| \pi^{\otimes_n})\right\}
	\leq 	  
	f \left( \mu \right) + F \left( \mu \right). 
	$$
We shall choose $\tau_n = \mu^{\otimes_n}$. 
Then we know that, by the 
law of large numbers,
we have the weak convergence 
$ i_n(\tau_n) \to \delta_\mu,$ so
	$$ \lim_{n \to \infty} 
	\mathbb E_{i_n \left(\tau_n \right)} 
	\left[  f  \right] = f(\mu) .$$
In addition, 
by using that 
$D (\tau_n \| \pi^{\otimes_n}) =
 n \, D(\mu \|  \pi)$ and
the {\it upper
limit assumption} \ref{upperlimitassumption}
 we get that
	$$\limsup_{n \to \infty} 
	 \left\{ \mathbb E_{i_n(\tau_n)} 
	 \left[ f \right] +  
	  \mathbb E_{\tau_n} 
	  \left[ W_n  \right] +
	 \frac{1}{n} D (\tau_n \| \pi^{\otimes_n})\right\}
	  \leq
	 f(\mu) + W(\mu) + D(\mu \| \pi) $$
completing the proof.

\end{proof}

\subsection{Proof of Theorem 
\ref{theorem: laplace}:
 Case of infinite $\beta$}

In this subsection we provide a proof for Theorem
 \ref{theorem: laplace} for the case of
 infinite $\beta$
by modifying the proof
used in the case of finite $\beta$.
Recall that from the definition 
of {\it Gibbs measure} (\ref{eq: gibbs}) and 
{\it free energy} (\ref{eq: free})
now we have
	$$d \gamma_n =
	e^{-n  \beta_n W_n} d\pi^{\otimes_n} ,
	\quad \text{and} \quad
	F = W .$$
where $\beta_n \to \infty$.

We first notice that a
{\it confining sequence} \ref{confiningsequence} 
satisfies an
a priori stronger property.

\begin{proposition}[Confining property of
the expected value of the energy]
\label{generalconfining}

Assume that $\{W_n\}_{n \in \mathbb N}$
is a stable \ref{stablesequence}
and confining \ref{confiningsequence} sequence
and take
a sequence of
probability measures $\{\chi_j\}_{j \in \mathbb N}$ on 
$\mathcal P(M)$, i.e. $\chi_j \in \mathcal P( \mathcal P(M))$.
Suppose there exists an increasing sequence
$\{n_j\}_{j \in \mathbb N}$ 
of natural numbers and a constant $C < \infty$
such that $\mathbb E_{\chi_j} \left[\tilde W_{n_j} \right] \leq C$
for every $j \in \mathbb N$ . Then
$\{\chi_j\}_{n \in \mathbb N}$ is 
relatively compact in 
$\mathcal P( \mathcal P(M))$.
\begin{proof}
The proof is left to the reader.
See for instance \cite[Lemma 2.1]{sandierserfaty}
for an idea or \cite[Proposition 3.4]{david}
and \cite[Proposition 7.4]{david} for a full proof.
\end{proof}

\end{proposition}	 
\ \\
Now we proceed with the proof of the theorem.
	 
\begin{proof}[Proof of Theorem 
\ref{theorem: laplace}: Case of infinite $\beta$]
Take $f:\mathcal P(M) \to \mathbb R$ bounded continuous.
By Subsection \ref{idea} 
about the idea of the proof we need to obtain
(\ref{eq: limit to prove}). We start proving
the {\bf lower limit bound}
$$
\liminf_{n \to \infty}  \inf_{\tau \in \mathcal P(M^n)}
	 \left\{ \mathbb E_{i_n(\tau)} \left[  f \right] + 
	 \mathbb E_\tau 
	  \left[ W_n  \right] + \frac{1}{n \beta_n}
	  D (\tau \| \pi^{\otimes_n})\right\} 
	  \geq 	  \inf_{\mu \in \mathcal P(M) } 
	  \{ f \left( \mu \right) + W \left( \mu \right)\} .    
	  $$
As in the proof used in the case of finite $\beta$
we want to see that
for every 
increasing sequence 
of natural numbers $\{n_j\}_{j \in \mathbb N}$
and choosing for each $j \in \mathbb N$ a
probability measure
$\tau_j \in \mathcal P(M^{n_j})$ we have	
	
\begin{equation} 
	\lim_{j \to \infty}\left\{	
	\mathbb E_{i_{n_j}(\tau_j)} \left[  f \right] + 
	 \mathbb E_{\tau_j}
	  \left[ W_{n_j}  \right] + 
	  \frac{1}{n_j \beta_{n_j}} 
	  D (\tau_j \| \pi^{\otimes_{n_j}})
	  \right\}
	\geq
	\inf_{\mu \in \mathcal P(M) } 
	\{ f \left( \mu \right) + 
	W \left( \mu \right)\}. 
\end{equation}		
where we can suppose that the limit
exists and it is finite.
As the entropy is non-negative we see that	
$\mathbb E_{\tau_j} \left[ W_{n_j}  \right]
= 
\mathbb E_{i_{n_j}(\tau_j)} \left[ \tilde W_{n_j}  \right]$
is a bounded sequence and, since 
$\{W_n\}_{n \in \mathbb N}$ is a  
{\it confining sequence} \ref{confiningsequence},
Proposition \ref{generalconfining} tells us
that $i_{n_j}(\tau_j)$ is 
relatively compact in 
$\mathcal P(\mathcal P(M))$. We continue as 
in the proof used in the case of finite $\beta$
where now $W$ is bounded from below
by the {\it regularity assumption} 
\ref{regularityassumption} and
because $\{W_n\}_{n \in \mathbb N}$
is a {\it stable sequence} \ref{stablesequence}.
\\
The proof of the {\bf upper limit bound}
follows the same reasoning as in 
the case of finite $\beta$.
Take $\mu \in \mathcal P(M)$.
Following the arguments used 
in the case of finite $\beta$ we can prove that
\begin{equation}
 \label{eq: upper limit bound without inf}
\limsup_{n \to \infty} 
\inf_{\tau \in \mathcal P(M^n)}
	 \left\{ \mathbb E_{i_n(\tau)} \left[  f \right] + 
	 \mathbb E_\tau 
	  \left[ W_n  \right] + 
	  \frac{1}{n \beta_n}
	  D (\tau \| \pi^{\otimes_n})\right\} 		
	  \leq 	\inf_{\mu \in \mathcal N}
	  \left\{  f \left( \mu \right) + 
	  W \left( \mu \right) \right\}
\end{equation}	 
where $\mathcal N$ was defined in (\ref{eq: nice}).
By the {\it regularity assumption} 
\ref{regularityassumption} we get
	$$\inf_{\mu \in \mathcal N}
	\left\{
	f \left( \mu \right) + W \left( \mu \right)
	\right\}
	=
	\inf_{\mu \in \mathcal P(M)}
	\left\{
	f \left( \mu \right) + W \left( \mu \right)
	\right\}
	$$
completing the proof.

\end{proof}

\subsection{Proof of Corollary \ref{cor: deviations}}

\begin{proof}[Proof of Corollary 
\ref{cor: deviations}]
We know that the large deviation principle is
equivalent to the Laplace principle for the sequence 
$i_n(\mathbb P_n)$ if
the rate function has compact level sets
(see
\cite[Theorem 1.2.1]{bdupuis} 
and 
\cite[Theorem 1.2.3]{bdupuis})

If $\beta < \infty$ this is the case because
the entropy has compact level sets
(see \cite[Lemma 1.4.3 (c)]{bdupuis}) and
$W$ is a lower semicontinuos function
bounded from below.
The lower semicontinuity of $W$ is a consequence
of the {\it lower and upper limit assumption, }
\ref{lowerlimitassumption} and
\ref{upperlimitassumption}.

If $\beta$ is infinite then there is no entropy
term and we can use that
$\{W_n\}_{n \in \mathbb N}$
is a {\it confining sequence}
\ref{confiningsequence},
and that $(\{W_n\}_{n \in \mathbb N},W)$ 
satisfies the {\it lower limit assumption } 
\ref{lowerlimitassumption}
and the {\it regularity assumption}
\ref{regularityassumption} 
to prove that $W$ has compact level sets.

Then we have to prove that,
for every bounded continuous
function $f: \mathcal P(M) \to \mathbb R$,
	$$\frac{1}{n\beta_n} 
	\log \,  \mathbb E_{i_n \left(\mathbb P_n \right)}
	 \left[ e^{- n \beta_n f} \right] 
	\,
	 \xrightarrow[\: n \to \infty \:]{}
	\,
	- \inf_{\mu \in \mathcal P(M) } \{ f \left( \mu \right) +
	 F \left( \mu \right) - \inf F \}$$
or, using the measures $\gamma_n$,

	$$\frac{1}{n \beta_n} \log \,  \int_{M^n} 
	 e^{- n \beta_n f \circ i_n} \frac{d \gamma_n}{Z_n} 
	 \,
	 \xrightarrow[\: n \to \infty \:]{}
	 \,  
	- \inf_{\mu \in \mathcal P(M) } \{ f \left( \mu \right) +
	 F \left( \mu \right) - \inf F \}.$$	 
This we can achieve
by using	 Theorem \ref{theorem: laplace} 
twice, for $f$ and for the zero function. 

\end{proof}

\begin{remark}[Other proof in the case of finite 
$\beta$]
When treating
the case
$\beta<\infty$, the
proof we are aware of 
is \cite{vortex}. It uses a ``quasi-continuity''
of the energy and it seems somewhat specific
to the logarithmic energy.
\end{remark}

\begin{remark}[Other proofs in 
the case of infinite $\beta$]
\label{remark: usual}
The proofs
that treat the case
$\beta = \infty$ usually follow
closely the approach we used
for the large deviation upper bound.
For the large deviation lower bound
they proceed as follows.
If $A$ is an open set of
$\mathcal P(M)$ and $\mu \in A$, they 
try to obtain
$$\liminf_{n \to \infty} \frac{1}{n \beta_n}
\log \int_{i_n^{-1}(A)} e^{-n \beta_n W_n} d\pi^{\otimes_n}	 
\geq - W(\mu).$$
For this, they search
pairwise
disjoint sets 
$B_1,...,B_n$ such that
$i_n(B_1 \times ... \times B_n) \subset A$
and such that
$\max_{B_1 \times... \times B_n} W_n
\simeq W(\mu)$. Then we may write
	$$
	 \int_{i_n^{-1}(A)} 
	 e^{-n \beta_n W_n} d\pi^{\otimes_n}	 
	\geq 
	\sum_{\sigma \in S_n}
	 \int_{B_{\sigma (1)} \times
	 ... \times B_{\sigma(n)}}
	  e^{-n \beta_n W_n} 
	 d\pi^{\otimes_n}	 
	 \geq
	 n! \pi(B_1)...\pi(B_n)
	 e^{-n \beta_n \max_{B_1 \times... \times B_n} W_n} 
	.$$
If we are able to choose those sets
such that $\pi(B_i) \geq \frac{C}{n}$ for
some $C$ independent of $n$ we can obtain,
using Stirling's formula,
	$$\liminf_{n \to \infty}
	\frac{1}{n \beta_n}\log(n!\pi(B_1)...\pi(B_n))
	\geq 
	\liminf_{n \to \infty}
	\frac{1}{n\beta_n}(\log n! - n \log n)
	\geq 0
	$$
and conclude by using that
$\lim_{n \to \infty}
\max_{B_1 \times... \times B_n} W_n
= W(\mu)$.

\end{remark}

\section{Applications}

\label{section: applications}

In this section we shall give the main applications
we are thinking of: Conditional Gibbs measure,
a Coulomb gas on a Riemannian manifold,
the known results of Coulomb gases in
Euclidean space
and the zeros of Gaussian random polynomials.

\subsection{Conditional Gibbs measure}

\label{subsection: conditional}

In this subsection we treat the case
of the Gibbs measure associated to a two-body interaction
but with some of the points conditioned to
be deterministic. 
We proceed by considering 
the deterministic points
as a background charge and
treat the interaction with this background
as some potential energy
that depends on $n$. More precisely, we use
the following more general setup. 

Let $\{\nu_n\}_{n \in \mathbb N}$ be a 
sequence of probability measures on 
a compact metric space $M$ that converges
to some probability measure
$\nu \in \mathcal P(M)$. 
Suppose we have a lower semicontinuous
function
$G^E:M \times M \to ( - \infty, \infty]$
that
shall be thought of as the interaction
 energy between the particles and the 
 environment
 and a symmetric lower semicontinuous function
$G^I:M \times M \to ( - \infty, \infty]$
that will be interpreted as the interaction
 energy between the particles. 
 More precisely
 we define two kinds of energy.\\
 {\it External potential energy.}
The probability measure $\nu_n$ will interact with
the $n$ particles via the external potential
$V_n: M \to \mathbb R$ defined by
$V_n(x) = \int_{M} G^E(x,y) d\nu_n(y)$.
This gives rise to the external energy
 $W^E_n: M^n \to ( - \infty, \infty]$
$$W^E_n(x_1,...,x_n) =
\frac{1}{n} 
	\sum_{i=1}^n V_n(x_i) $$
with a macroscopic external energy
$W^E : \mathcal P(M) \to ( - \infty, \infty]$
 	$$W^E(\mu) = \int_{M \times M} 
 	G^E(x,y)\,  d\mu(x) d\nu(y). $$
{\it Internal potential energy.} 
For each $n$ we shall think of $n$ particles
interacting with the two-particle potential $G^I$.
This would give rise to 
an internal energy  
$W^I_n : M^n \to ( - \infty, \infty]$
$$W^I_n(x_1,...,x_n) =
	\frac{1}{n^2}\sum_{i < j}^n G^I(x_i,x_j) $$
and a macroscopic internal energy
$W^I : \mathcal P(M) \to ( - \infty, \infty]$
	 $$W^I(\mu) = 
	 \frac{1}{2}\int_{M \times M}
	  G^I(x,y)\,  d\mu(x) d\mu(y) .$$
{\it Total potential energy.} 
For each $n$ we define
	$$W_n = W^E_n + W^I_n \ \ \ \mbox{ and } \ \ \ 
	W = W^E + W^I. $$

Then, it is not hard to see that
$\{W_n\}_{n \in \mathbb N}$ is a
{\it stable sequence} \ref{stablesequence} 
and
$W$ is a lower semicontinuous function.
The example of a conditional
Gibbs measure
can be obtained essentially by choosing as $\nu_n$
the empirical measure of some points and
$G^I = G^E$. So, a particular case of the
next theorem is a Coulomb gas
conditioned to all but an increasing number of points.

\begin{theorem}[Varying environment]
Suppose that
	$x \mapsto \int_M G^E(x,y) d\nu(y)$
is continuous. 

\label{theorem: conditionalgas}

Let
	$$\tilde{\mathcal N} =
	\left\{ \mu \in \mathcal P(M) :\ 
	D(\mu\| \pi) < \infty \mbox{ and }
	y \mapsto \int_M G^E(x,y) d\mu(x) 
	\mbox{ is continuous}
	\right\}$$
and suppose that for every 
$\mu \in \mathcal P(M)$ such that $W^I(\mu) < \infty$
there exists a sequence 
$\{\mu_n\}_{n \in \mathbb N}$ of probability
measures
in $ \tilde{\mathcal N} $ such that $\mu_n \to \mu$
and $ W^I(\mu_n) \to  W^I(\mu)$.

Then 
$W$ is the zero temperature macroscopic limit of
$\{W_n\}_{n \in \mathbb N}$.
In particular, if 
we choose $\beta_n \to \infty$, 
Theorem \ref{theorem: laplace}
and Corollary \ref{cor: deviations} may
be applied for $(\{W_n\}_{n \in \mathbb N},W)$.

\begin{proof}
Let us prove
the {\it lower limit assumption} 
\ref{lowerlimitassumption}.

{\bf Lower limit assumption \ref{lowerlimitassumption}.} 
By Proposition \ref{proposition: k-body}, 
we already know that 
$(\{W^I_n\}_{n \in \mathbb N},W^I)$
satisfies the {\it lower limit assumption} 
\ref{lowerlimitassumption}.
We only 
need to check this for $(\{W^E_n\}_{n \in \mathbb N},
W^E)$.

If $\mu_n = i_n(x_1,...,x_n)$ then
	$$\tilde W^E_n(\mu_n) = \int_M V_n d\mu_n =
					\int_{M \times M}
					 G^E(x,y) d\mu_n(x) d\nu_n(y),$$
where $\tilde W^E_n$ is defined in 
(\ref{eq: extension}).		
So, the {\it lower limit assumption}
\ref{lowerlimitassumption}
is a consequence of the lower semicontinuity of $G^E$.

{\bf Regularity assumption 
\ref{regularityassumption}.}
To prove the {\it regularity assumption} 
\ref{regularityassumption}
we take
$\mu \in \mathcal P(M)$ such that
$W(\mu) < \infty$. 
Then $W^I(\mu) < \infty$.
By hypothesis, 
we know that there exists
a sequence $\{\mu_n\}_{n \in \mathbb N}$ of probability
measures
in $ \tilde{\mathcal N} $ such that $\mu_n \to \mu$
and $ W^I(\mu_n) \to  W^I(\mu)$.
As $x \mapsto \int_M G^E(x,y) d\nu(y)$ is continuous
we also have that
$W^E(\mu_n) \to W^E(\mu)$.
So, $W(\mu_n) \to W(\mu)$.
 
We have to prove that the sequence we chose is in the set
$\mathcal N$ defined in (\ref{eq: nice}) by
$$\mathcal N = \left\{ \mu \in \mathcal P(M) :\
D(\mu \| \pi) < \infty \mbox{ and }
\limsup_{n \to \infty}
 \mathbb E_{\mu^{\otimes_n}}[W_{n}] \leq W(\mu)
 \right\},$$
i.e. we need to see that
$\tilde{\mathcal N} \subset \mathcal N$.

Let $\mu \in \tilde{\mathcal N}$. Then
\begin{align*}
	\mathbb E_{\mu^{\otimes_n}}[W^E_n] 
	&=\mathbb E_{\mu}[V_n]			\\
	&=\int_M
	\left(\int_M G^E(x,y) d\nu_n(y)\right)
	d\mu(x)					 		\\
	&=\int_M
	\left(\int_M G^E(x,y) d\mu(x)\right)
	d\nu_n(y). 
\end{align*}
So, as $y \mapsto 
\int_M G^E(x,y) d\mu(x)$ is continuous,
we get 
	$$\mathbb E_{\mu^{\otimes_n}}[W^E_n] 
	\,
	\xrightarrow[\: n \to \infty \:]{}
	\,	
	 \int_{M \times M}
	  G^E(x,y) d\nu(y)  d\mu(x) = W^E(\mu).$$
By Proposition \ref{proposition: k-body} we already know that 
$\lim_{n \to \infty} \mathbb E_{\mu^{\otimes_n}}[W^I_n] =
 W^I(\mu) $
and then
$\mu \in \mathcal N$.

\end{proof}
\end{theorem}

For the sake of completeness
we treat the case of 
a Coulomb gas conditioned
to all points but a finite fixed 
number of them. Again, by considering
the deterministic points
as a background charge we can use the following
more general framework.
Suppose we have two compact metric spaces
$M$ and $N$, a probability measure $\Pi$
on $N$ and two
lower semicontinuous functions 
$G^E:N \times M \to ( - \infty, \infty]$
and
$G^I: N \to (-\infty, \infty]$.
Let $\{\nu_n\}_{n \in \mathbb N}$ be a 
sequence of probability measures on 
$M$ that converges
to some probability measure
$\nu \in \mathcal P(M)$.
We will consider one particle in $N$ interacting
with the environment via $G^E$, i.e.
via a potential energy $V_n:N \to ( - \infty, \infty]$
defined by $V_n(x) = \int_M G^E(x,y) d\nu_n(y)$.
This particle
will also have a self-interaction
given by $\lambda_n G^I$ where 
$\{\lambda_n\}_{n \in \mathbb N}$
is a sequence that converges to zero.
The case of a Coulomb gas conditioned to all
but $k$ particles may be obtained by essentially 
taking
 $N = M^k$, 
$\Pi = \pi^{\otimes_k}$,
$G^I(x_1,...,x_k) = \sum_{i<j}G(x_i,x_j)$,
$G^E ((x_1,...,x_k),y) = \sum_{i=1}^k G(x_i, y)$,
 $\lambda_n = \frac{1}{n}$
 and $\nu_n$ as the empirical measure
 of the deterministic particles.

\begin{theorem}
[A particle in a varying environment]
Suppose that $V:N \to \mathbb R$
defined by $V(x) = \int_M G^E(x,y)d\nu(y)$
is (bounded and) continuous.
Let
	$$\tilde{\mathcal N} =
	\left\{ \mu \in \mathcal P(N) :\ 
	D(\mu\| \Pi) < \infty,
	\int_N G^I d\mu< \infty \mbox{ and }
	y \mapsto \int_N G^E(x,y) 
	d\mu(x) \mbox{ is continuous}
	\right\}$$
and suppose that for every $z \in N$ 
there exists a sequence 
of probability measures $\{\mu_n\}_{n \in \mathbb N}$
in $\tilde{\mathcal N}$ such that $\mu_n \to \delta_z$.
Take a sequence of
non-negative numbers $\{\beta_n\}_{n \in \mathbb N}$
such that $\beta_n \to \infty$
and define the measures $\gamma^c_n$ by
	$$d\gamma^c_n = e^{-\beta_n (V_n + \lambda_n G^I)}
	 d\Pi.$$
Then, we have
the following Laplace principle. 	
For every (bounded) continuous function
$f: N \to \mathbb R$  
	$$\frac{1}{\beta_n} 
	 \log \,  \int_{N} 
	 e^{- \beta_n  f  }
	 d \gamma^c_n
	\,
	\xrightarrow[\: n \to \infty \:]{}
	\,
	 - \inf_{x \in N}
	\left\{ f (x) +
		V(x)\right\}.	$$
	
\begin{proof}
We use Lemma \ref{lemma: legendre} to write
	$$\frac{1}{\beta_n} 
	 \log \,  \int_{N} 
	 e^{- \beta_n  f  }
	 d \gamma^c_n
	 =
	 - \inf_{\mu \in \mathcal P(N)}
	 \left\{ 
	 	 \mathbb E_\mu \left[ f \right]
	 	 +
		\mathbb E_\mu \left[V_n \right]
		+ \lambda_n 		 
		\mathbb E_\mu \left[G^I \right]
		+ \frac{1}{\beta_n}D(\mu\|\Pi)
	 \right\}.$$
Following the same ideas used in the proofs
of Theorem \ref{theorem: laplace}
and Theorem \ref{theorem: conditionalgas} we get	
	$$\liminf_{n \to \infty}
	\inf_{\mu \in \mathcal P(N)}
	 \left\{ 
	 	\mathbb E_\mu \left[ f \right] +
		\mathbb E_\mu \left[V_n \right]
		+ \lambda_n 		 
		\mathbb E_\mu \left[G^I \right]
		+ \frac{1}{\beta_n}D(\mu\|\Pi)
	 \right\}
	 \geq \inf_{\mu \in \mathcal P(N)}
	 \left\{ 
	 	 \mathbb E_{\mu}[f] + 
	 	 \mathbb E_\mu[V]
		\right\} $$
and
	 	$$\limsup_{n \to \infty}
	\inf_{\mu \in \mathcal P(N)}
	 \left\{ 
	 	\mathbb E_\mu \left[ f \right]
	 	+
		\mathbb E_\mu \left[V_n \right]
		+ \lambda_n 		 
		\mathbb E_\mu \left[G^I \right]
		+ \frac{1}{\beta_n}D(\mu\|\Pi)
	 \right\}
	 \leq \inf_{\mu \in \tilde{\mathcal N}}
	 \left\{
	 	 \mathbb E_{\mu}[f] 
	 	 + \mathbb E_\mu[V]
		\right\} .$$
We shall think of $N$ as included
in $\mathcal P(N)$
by the application $z \mapsto \delta_z$. 
Then, by the continuity of $V$ and $f$ and as
we are assuming that elements
of $N$ are approximated by
elements of $\tilde{\mathcal N}$ we know that
	$$\inf_{\mu \in \tilde{\mathcal N}} 
	 \left\{ \mathbb E_\mu[f]
	 +\mathbb E_{\mu}[V] \right\}	
	 =
	 \inf_{\mu \in \tilde{\mathcal N} \cup N}
	 \left\{ \mathbb E_{\mu}[f]
	 +\mathbb E_\mu[V]
	 \right\}	.$$	 
As the infimum is achieved in $N$ we get
	$$\inf_{\mu \in \tilde{\mathcal N}} 
	 \left\{ \mathbb E_\mu[f]
	 +\mathbb E_{\mu}[V] \right\}	
	 =\inf_{x \in N} 
	 \left\{ f(x) + V(x)\right\}	
	 =
	 \inf_{\mu \in \mathcal P(N)} 
	 \left\{ \mathbb E_\mu[f]
	 +\mathbb E_{\mu}[V] \right\}$$
concluding the proof.	 
\end{proof}	
	
\end{theorem}

\subsection{A Coulomb gas on a Riemannian manifold}
\label{subsection: riemannian}

Let $(M,g)$ be a compact oriented $n$-dimensional
Riemannian manifold without boundary where
$g$ denotes the Riemannian metric.
We shall define a 
continuous function
$G:M \times M \to ( - \infty, \infty]$ naturally
associated to the Riemannian structure of $M$.
This function along with
the normalized volume form $\pi$ of $(M,g)$
will allow us to define the Gibbs measures
$\gamma_n$ of (\ref{eq: gibbs}) and
will put us in the context of Theorem
\ref{theorem: laplace}.

For this we establish some notation.
A signed measure $\Lambda$ will be called a
differentiable signed measure 
if it is given by an $n$-form or
equivalently if it has a differentiable
density with respect to $\pi$.
From now on we shall identify $\Omega^n(M)$ with the
space of differentiable signed measures.
Denote by $\Delta: C^\infty(M) \to \Omega^n(M)$ 
the Laplacian
operator, i.e. $\Delta = d*d$ where $*$
is the Hodge star operator or, equivalently,
$\Delta f = \nabla^2 f\, d\pi$ where $\nabla^2$
is the Laplace-Beltrami operator.
The function $G$
we will be interested in
is given by
the following result.

\begin{proposition}[Green function]
\label{proposition: green}

Take any differentiable signed measure $\Lambda$.
Then, there exists
a symmetric continuous function
$G:M \times M \to ( - \infty, \infty]$ such that
for every $x \in M$ the function
$G_x: M \to ( - \infty, \infty]$ defined
by $G_x(y) = G(x,y)$ is integrable with respect to
$\pi$ and
	$$\Delta G_x = -\delta_x + \Lambda.$$
More explicitly, the previous equality can be written
as follows. For every $f \in C^\infty(M)$
we have
	$$\int_M G_x \, \Delta f = -f(x) + 
	\int_M f  d \Lambda.$$
Such a function will be called a Green
function associated to $\Lambda$.
Furthermore $G$ is integrable with respect 
to $\pi \otimes \pi$. If $\mu$ is a 
differentiable signed measure
then $\psi: M \to \mathbb R$
defined by $\psi(x)=\int_M G(x,y)d\mu(y)$ belongs
to $C^\infty(M)$ and
	$$\Delta \psi = - \mu +  \mu(M) \Lambda.$$

In particular, we can get that
$G$ is bounded from below,
$\int_M G_x d\Lambda$ does not depend on $x \in M$
and a Green function
associated to $\Lambda$  
is unique up to an additive constant.

\begin{proof}

This result is well known if $\Lambda = \pi$.
See for instance \cite[Chapter 4]{aubin}.
Then if $H$ is a Green function associated to $\pi$
we define
$\phi \in C^{\infty}(M)$ by
$\phi(x) = \int_M H(x,y) d\Lambda(y)$ 
and 
the function
 $G: M \times M \to ( - \infty, \infty] $
given by $G(x,y) = H(x,y) - \phi(x) - \phi(y)$ 
is a Green function associated to $\Lambda$.  
 
\end{proof}
\end{proposition}
\ \\
We fix
a differentiable signed measure $\Lambda$.
For simplicity we choose the Green function $G$
associated to $\Lambda$ that satisfies
$\int_M G_x d\Lambda = 0$ for every $x \in M$.
Define
$W_n: M^n \to \mathbb R \cup \{\infty\}$
by $W_n(x_1,...,x_n) 
= \frac{1}{n^2}\sum_{i<j}^n G(x_i,x_j)$
and
$W: \mathcal P(M) \to ( - \infty, \infty] $
by $W(\mu) = 
\frac{1}{2}\int_{M \times M} G(x,y) d\mu(x) d\mu(y)$.
Because $G$ is bounded from below
and lower semicontinuous
we may apply 
Proposition \ref{proposition: k-body} about the k-body interaction. In particular,
we obtain that
$\{W_n\}_{n \in \mathbb N}$ is a
{\it stable sequence} \ref{stablesequence},
$W$ is lower semicontinuous and 
$(\{W_n\}_{n \in \mathbb N}, W)$
satisfies the
{\it lower limit assumption}
\ref{lowerlimitassumption}
and the {\it upper limit assumption}
\ref{upperlimitassumption}.

We can prove a strong form of the regularity assumption
for $W$.

\begin{proposition}[Regularity property
of the Green energy]

\label{proposition: regularity of the green energy}
 
Let $\mu \in \mathcal P(M)$. There exists
a sequence $\{\mu_n\}_{n \in \mathbb N}$
of differentiable probability measures such that
$\mu_n \to \mu$ and 
$W(\mu_n) \to W(\mu)$.

\begin{proof}
We can assume $W(\mu) < \infty$, otherwise
any sequence $\{\mu_n\}_{n \in \mathbb N}$
of differentiable probability measures such that
$\mu_n \to \mu$
will satisfy $W(\mu_n) \to W(\mu)$ 
due to the lower semicontinuity of $W$.

Using the proof of \cite[Lemma 3.13]{cbeltran} 
for the case of probability measures we know
that the result is true for
the Green function $H$ associated to $\pi$.
For general $\Lambda$,
take $\phi \in C^\infty(M)$ 
defined by
$\phi(x) = \int_M H(x,y) d\Lambda(y)$ as in 
the proof of Proposition \ref{proposition: green}.
Then $G: M \times M \to ( - \infty, \infty]$
given by $G(x,y) = H(x,y) - \phi(x) - \phi(y)$
is a Green function for $\Lambda$ and
for every $\mu \in \mathcal P(M)$ we have
	$$\int_{M \times M} G(x,y) d\mu(x) d\mu(y)
	= \int_{M \times M} H(x,y) d\mu(x) d\mu(y)
	- 2\int_{M} \phi \, d\mu.$$
From this relation and the result for $H$ we get
the result for $G$.

\end{proof}
 
\end{proposition}
\ \\
Then, $(\{W_n\}_{n \in \mathbb N}, W)$ is a nice model
where Theorem \ref{theorem: laplace}
and the results of Subsection
\ref{subsection: conditional} can be used.

\begin{corollary}[Macroscopic limit]
$W$ is the zero temperature macroscopic limit
and the positive temperature macroscopic limit
of $\{W_n\}_{n \in \mathbb N}$, i.e.
  $(\{W_n\}_{n \in \mathbb N}, W)$
 satisfies all the conditions
of Theorem \ref{theorem: laplace}.
Additionally the results of Subsection
\ref{subsection: conditional} about
the Conditional Gibbs measure may be applied.

\end{corollary}
\ \\

Now we shall enunciate a theorem that
is our main motivation for choosing this model.
Remember the definitions of $i_n$,
 (\ref{eq: inclusion}),
and $\mathbb P_n$, 
(\ref{eq: gibbs probability finite beta}).
Let $\{X_n\}_{n \in \mathbb N}$
be a sequence of random variables
taking values in $\mathcal P(M)$ such that, for every 
$n \in \mathbb N$,
$X_n$ has law $i_n (\mathbb P_n) $.
By studying the minimizers of the
{\it free energy} $F$ defined in (\ref{eq: free})
 we can understand
the possible limit points of 
$\{X_n\}_{n \in \mathbb N}$.
In particular, if $F$ attains its minimum
at a unique probability measure
$\mu_{eq}$, we get
	$$X_n     
	 \xrightarrow[\: n \to \infty \:]{a.s.}
	 \mu_{eq}.$$
This is a consequence of Borel-Cantelli lemma
and the large deviation principle 
in Corollary \ref{cor: deviations}.

We specialize
to the case of dimension two
and finite $\beta$
because the minimizer of $F$
has
a nice geometric meaning in this case. 

\begin{theorem}[Minimizer of the free energy]
\label{theorem: meanfieldequation}
Let $\rho$ be a strictly positive differentiable
function such that
\begin{equation} \label{eq: meanfield}
	 \Delta \log \rho  = 
	\beta \, \mu_{eq} - \beta\Lambda
\end{equation}
where $\mu_{eq}$ denotes
the probability measure defined by
$d\mu_{eq} = \rho \, d\pi$ (see  \cite{meanfield}
for the existence).
Then $F(\mu_{eq}) < F(\mu)$ 
for every $\mu \in \mathcal P(M)$ different
from $\mu_{eq}$.
In particular, there exists only
one strictly positive differentiable function
that satisfies (\ref{eq: meanfield}).
\end{theorem}

\begin{remark}[Scalar curvature relation]
The motivation
for studying a $2$-dimensional manifold
is that $\mu_{eq}$ has a nice geometrical
interpretation if we choose 
adequate
$\Lambda$ and $\beta$.

We shall suppose that $\chi(M)$, 
the Euler characteristic of $M$, is different from
 zero.
If $\bar g$ is any metric, we denote by $R_{\bar g}$
the scalar curvature of $\bar g$.
Choose
	$$d\Lambda = \frac{ R_g \, d  \pi} 
	{4\pi \chi (M)}.$$		
It can be seen that if $\bar g = \rho g$,
where $\int_M \rho \, d\pi = 1$, then
 	$$\Delta \log \rho = R_g d\pi - R_{\bar g} 
 	\rho \, d\pi .$$
With this identity we can prove that 
$\rho$ is a solution to
	\begin{equation*}
	R_{\bar g} =
	\left( 4\pi \chi(M) + \beta\right) 
 	R_g \rho^{-1} - \beta 			
	\end{equation*}
where  	$\bar g = \rho \, g$ if and only if
$\rho$ is a solution to
 	\begin{equation*} 
	\Delta \log \rho  = 
	\beta \, \mu_{eq} - \beta\Lambda 
	\end{equation*}
where 	$d\mu_{eq} = \rho \, d\pi$.	
 In particular, if $\chi(M) <0$ and
 $\beta = -4 \pi \chi(M)$
 then
 $\bar g$ satisfies
	$$R_{\bar g} =
		4 \pi \chi(M), $$
i.e. $\bar g$ is a metric with constant curvature.
In other words, if
 $\beta = -4 \pi \chi(M)$,
 the empirical measure converges
almost surely to the volume form
of the constant curvature metric conformally
equivalent to the chosen metric.
\end{remark}

The proof of Theorem \ref{theorem: meanfieldequation}
will be based
on the fact that $F$ is strictly convex
and
that we can calculate its derivative.
We begin by proving its convexity.

\begin{proposition}[Convexity of $W$]
$W$ is convex.
\begin{proof}

To prove the convexity it is enough to show that
for every $\mu, \nu \in \mathcal P(M)$ 
\begin{equation}
\label{eq: convexity inequality}
	\frac{1}{2} W(\mu)+
	\frac{1}{2} W(\nu)
	\geq
	W \left(\frac{1}{2}\mu + \frac{1}{2} \nu \right)
\end{equation}
due to the lower semicontinuity of $W$.	
If $\mu$ and $\nu$
are differentiable probability measures this is equivalent to
	$$ \int_M \| \nabla (f - g)\|^2 \, d \pi \geq 0 $$
where
	$f(x) = \int_{M} G (x,y) \, d \mu(y)$
	 and
	$g(x) = \int_{M} G (x,y)\, d \nu(y) $.
For general $\mu$ and $\nu$ we can conclude
using
Proposition
\ref{proposition: regularity of the green energy},
and taking lower limits in
the inequality 
(\ref{eq: convexity inequality}) 
for differentiable
probability measures.
\end{proof}

\end{proposition}
\ \\
As $D(\cdot \| \pi)$ is strictly convex 
(see \cite[Lemma 1.4.3]{bdupuis}) we obtain that
the {\it free energy} $F$
of parameter $\beta < \infty$ is strictly convex.

Now we calculate the derivative 
of $W$ and the entropy
at $\mu_{eq}$.

\begin{lemma} [Derivative of $W$ and the
entropy]
\label{lemma: derivative}

Let $\mu$ be
any probability measure different from $\mu_{eq}$ 
such that $F(\mu) < \infty$.
Define
	$$\mu_t = t\mu + (1-t)\mu_{eq},\ \ t \in [0,1].$$\\
Then,
$ W(\mu_t)$ and $D(\mu_t \| \pi)$
are differentiable at $t = 0$, and
\begin{equation} \label{eq: derivative energy}
\hspace*{-4cm}
	\frac{d}{dt}  W(\mu_t)|_{t = 0} 				= 
		\int_{M \times M} G(x,y) \, d\mu_{eq} (x) \,
		\left(d \mu(y) - d \mu_{eq}(y)\right)\, ,
\end{equation}		

\begin{equation} \label{eq: derivative entropy}
\hspace*{-6.5cm}
	\frac{d}{dt} D(\mu_t \| \pi)|_{t = 0}		= 		
		\int_{M} \log  \rho(y) 
			\left(d \mu(y) - d \mu_{eq}(y) \right).
\end{equation}

\begin{proof}
To get (\ref{eq: derivative energy}) we just notice that
$W(\mu_t)$ is a polynomial of degree $2$ and
to obtain
(\ref{eq: derivative entropy}) we use
the monotone convergence theorem
as said for instance in
\cite[Proposition 2.11]{bermanThermodynamicalformalism}.

\end{proof}

\end{lemma}
	
And now we are ready to finish the proof
of Theorem 
\ref{theorem: meanfieldequation}.

\begin{proof}[Proof of Theorem 
\ref{theorem: meanfieldequation}]
As in Lemma \ref{lemma: derivative}, let $\mu$ be
any probability measure different from $\mu_{eq}$ 
such that $F(\mu) < \infty$ and
define
	$$\mu_t = t\mu + (1-t)\mu_{eq},\ \ t \in [0,1].$$
Multiply (\ref{eq: meanfield})
by $G(x,y)$ and integrate in one variable
to get

	$$ -\log \rho(y) + \int_M \log \rho(x)
	 \, d \Lambda(x)
		=\beta \int_M G(x,y)\rho(x) \, d \pi(x) .$$
Then, we have that

\begin{align*}
	\frac{d}{dt} & F(\mu_t)|_{t = 0} \\
	& = 
	\int_{M\times M} G(x,y) d \mu_{eq}(x)
		\, \left(d \mu(y) - d \mu_{eq}(y) \right)+
			\frac{1}{\beta}
			\int_{M \times M} \log  \rho(y) 
			\left(d \mu(y) - d \mu_{eq}(y) \right) \\
	&= \int_M
	\left(\int_M G(x,y) \rho(x) \, d \pi(x) +
	\frac{1}{\beta}\log \rho(y)  \right)
		\left(d \mu(y) - d \mu_{eq}(y)\right)		\\
	&= \frac{1}{\beta}\int_M
	\left(\int_M \log  \rho(x)\, d \Lambda(x) \right)
		\left(d \mu(y) - d \mu_{eq}(y) \right)		\\
	&= 	\frac{1}{\beta}
	\left(\int_M \log \rho(x)\, d \Lambda(x) \right)
		\left(\int_M
		\left( d \mu(y) - d \mu_{eq}(y)\right) 
		\right)= 0.
\end{align*}		
This implies, due to the strict convexity of $F(\mu_t)$
in $t$, that
	$$ F(\mu_{eq}) < F(\mu).$$ 
	
\end{proof}

\subsection{Usual Coulomb gases}

In this subsection we provide
different proofs to the large deviation
principles associated to Coulomb gases
studied in
 \cite{hardy} 
and \cite{adupuis}. These
models are usually motivated as 
describing the laws of eigenvalues of
some random matrices and has as particular cases 
the models studied in \cite{ldpwigner},
\cite{hiaiunitary}, \cite{hiaigaussian}
and
 \cite{chafai}.
We may see \cite{guionnet}
for an introduction to random matrices.
We would like to remark that
the model studied in
\cite{zeitouni} may be treated by the same
methods but does not fall directly in the
regime of
application of Theorem \ref{theorem: laplace}.

Suppose that $l$ is a not necessarily finite measure
on the Polish space $M$.
Let $V: M \to ( - \infty, \infty] $ 
and
$G: M \times M \to ( - \infty, \infty] $
be lower semicontinuous functions with $G$ symmetric
and such that $(x,y) \mapsto G(x,y) + V(x)+V(y)$
is bounded from below.
Define $H_n : M^n \to ( - \infty, \infty] $ by
	$$ H_n(x_1,...,x_n)
	= \sum_{i < j}^n G(x_i, x_j) +
	n \sum_{i =1}^n V(x_i)$$
and $W: \mathcal P(M) \to ( - \infty, \infty] $
by
	$$
	W(\mu) = \frac{1}{2}\int_{M\times M}
	\left(G(x,y) + V(x) + V(y) \right) 
	 d\mu(x) d\mu(y).
	$$
Take a sequence $\{\beta_n\}_{n \in \mathbb N}$
such that $\beta_n \to \infty$ and
let  $\gamma_n$ be the Gibbs measure defined by	
	$$d \gamma_n =
	e^{-\frac{\beta_n}{n} H_n} d l^{\otimes_n} .$$
We shall give some hypotheses that imply
that $\gamma_n$ satisfies
a Laplace principle.

The first example is related to
\cite{hardy}. More precisely,
if
we choose $G(x,y) = -\beta \log \|x - y\|$,
condition
$(1.7)$ of \cite{hardy} implies the first
three conditions of 
the following theorem 
(see the proof of Proposition 
\ref{proposition: gibbspolynomial}
for an idea) and the last condition is a consequence
the nature of the logarithmic interaction and
the required continuity of $V$ in \cite{hardy}.
We remark
that there is a slight typo in 
\cite{hardy}: we should require $\beta' > 2$
in
dimension two.

\begin{theorem}[Weakly confining case]
\label{theorem: weakly}

Take $\beta_n = n$. Suppose that

$\bullet$ $\int_M e^{-V} dl < \infty$,

$\bullet$ the function
$(x,y) \mapsto G(x,y) + V(x) + V(y)$
is bounded from below,

$\bullet$ $G(x,y) + V(x) + V(y) \to \infty$ when
$x, y \to \infty$ at the same time, and

$\bullet$
for every $\mu \in \mathcal P(M)$
such that $W(\mu) < \infty$,
there exists a sequence $\{\mu_n \}_{n \in \mathbb N}$
of probability measures absolutely continuous
with respect to $l$ such that
$\mu_n \to \mu$ and $W(\mu_n) \to W(\mu)$.

Then, for every bounded continuous function
$f: \mathcal P(M) \to \mathbb R$
we have
	$$\frac{1}{n^2} \log \,  \int_{M^n} 
	 e^{- n^2 f \circ i_n} d \gamma_n  
	 \,
	  \xrightarrow[\: n \to \infty \:]{}
	 \,
	- \inf_{\mu \in \mathcal P(M) } 
	\{ f \left( \mu \right) +
	 W \left( \mu \right) \}.$$	 
	 
\begin{proof}	 
Assume $\int_M e^{-V} dl = 1$ 
for simplicity. We notice that
	$$d\gamma_n =
	e^{-
	\left(  \sum_{i < j}^n G(x_i, x_j) +
	(n - 1) \sum_{i =1}^n V(x_i) 
	 \right)}
	d(e^{-V} l)^{\otimes_n}.$$
If we define
	$$\tilde G(x,y) = G(x,y) + V(x) + V(y)$$
and
	$$W_n(x_1,...,x_n) = 
	\frac{1}{n^2}\sum_{i < j}^n \tilde G(x_i, x_j)$$
we have
	$$d \gamma_n =
	e^{-n^2 W_n }
	d(e^{-V} l)^{\otimes_n}.$$
We now prove that $\{W_n\}_{n \in \mathbb N}$ 
satisfies
the conditions necessary to apply 
Theorem \ref{theorem: laplace}.	

{\bf Lower and upper limit assumption, 
\ref{lowerlimitassumption} and 
\ref{upperlimitassumption}}.
By hypotheses, $\tilde G$ is lower semicontinuous
and bounded from below.
We can apply Proposition \ref{proposition: k-body} to get
that $\{W_n\}_{n \in \mathbb N}$
is a {\it stable sequence} \ref{stablesequence}
 and
that $(\{W_n\}_{n \in \mathbb N},W)$ satisfies
the {\it lower limit assumption} 
\ref{lowerlimitassumption}
and the {\it upper limit assumption} 
\ref{upperlimitassumption}.

{\bf Regularity assumption 
\ref{regularityassumption}}.
Since $(\{W_n\}_{n \in \mathbb N},W)$
satisfies the {\it upper limit assumption} 
\ref{upperlimitassumption},
the {\it regularity assumption} 
\ref{regularityassumption} does not depend
on $\{W_n\}_{n \in \mathbb N}$
and we can use
Proposition \ref{proposition: kbodyregularity}.
Take $\mu \in \mathcal P(M)$ such that
$W(\mu) < \infty$. Then, by hypothesis,
there exists a sequence $\{\mu_n\}_{n \in \mathbb N}$
of probability measures absolutely continuous
with respect to $l$ such that $\mu_n \to \mu$ and
$W(\mu_n) \to W(\mu)$. As $W(\mu) < \infty$
we can assume $W(\mu_n) < \infty$ for every
$n \in \mathbb N$. Fix $n \in \mathbb N$. 
We want to prove that $\mu_n$ 
is absolutely continuous with respect to
the measure defined by $e^{-V} dl$.
For this it is enough
to notice that
$\mu_n (\{x \in M : \ V(x) = \infty\}) = 0$.
We can see that the set 
$\{(x,y) \in M \times M: \ V(x) = \infty
\mbox{ and } V(y) = \infty\}$
is included in the set
$\{(x,y) \in M \times M :
\ G(x,y) + V(x) + V(y) = \infty\}$. The latter
has zero measure because $W(\mu) < \infty$ and
we conclude by the definition of product measure.
 
{\bf Confining sequence \ref{confiningsequence}}.
Using that
$\tilde G(x,y) \to \infty$ when
$x, y \to \infty$ at the same time
and Proposition \ref{proposition: kbodyconfining}
we get that $\{W_n\}_{n \in \mathbb N}$ is a {\it confining sequence} \ref{confiningsequence}.

We can finally
apply Theorem \ref{theorem: laplace}.
\end{proof}

\end{theorem}
\ \\
The second example is 
related to the article this
work is inspired on, i.e.  \cite{adupuis}.
More precisely, Assumptions C1-C3 of 
\cite[Theorem 1.6]{adupuis} imply the conditions of the following theorem.
We remark that there is a slight typo
in \cite{adupuis}: Assumption A should
be changed by any weaker assumption that guarantees
the finiteness of the Gibbs measures.

\begin{theorem}[Strongly confining case]
\ \\
Suppose that

$\bullet$ 
There exists $\xi > 0$
such that $\int_M e^{- \xi V} dl < \infty$,

$\bullet$
$V$ is bounded from below,

$\bullet$
there exists $\epsilon \in [0,1)$ such that
$(x,y) \mapsto G(x,y) + \epsilon V(x) + \epsilon V(y)$
is
bounded from below, 

$\bullet$
the function
$G(x,y) +  V(x) +  V(y)$
tends to infinity when $x,y \to \infty$
at the same time, and

$\bullet$
for every $\mu \in \mathcal P(M)$
such that $W(\mu) < \infty$,
there exists a sequence $\{\mu_n \}_{n \in \mathbb N}$
of probability measures absolutely continuous
with respect to $l$ such that
$\mu_n \to \mu$ and $W(\mu_n) \to W(\mu)$.

Then, for every bounded continuous function
$f: \mathcal P(M) \to \mathbb R$
we have
	$$\frac{1}{n\beta_n} \log \,  \int_{M^n} 
	 e^{- n\beta_n f \circ i_n} d \gamma_n  
	 \,
	  \xrightarrow[\: n \to \infty \:]{}
	 \,
	- \inf_{\mu \in \mathcal P(M) } 
	\{ f \left( \mu \right) +
	 W \left( \mu \right) \}.$$	 

\begin{proof}

We can assume $\int_M e^{- \xi V} dl = 1$ for 
simplicity.
Then we can write
	$$d \gamma_n =
	e^{-\frac{\beta_n}{n}
	\left(  \sum_{i < j}^n G(x_i, x_j) +
	\left(n - \frac{n}{\beta_n} \xi \right) 
	\sum_{i =1}^n V(x_i) 
	 \right)}
	d (e^{-\xi V} l)^{\otimes_n}.$$
which may only make sense
for $n$ large enough
due to some positive and negative 
infinities. If we define
	$$G^n(x,y) = G(x,y) + \frac{1}{n-1}
	\left(n - \frac{n}{\beta_n} \xi \right) V(x) +
	\frac{1}{n-1}
	\left(n - \frac{n}{\beta_n} \xi \right) V(y) $$
and
	$$W_n(x_1,...,x_n) = 
	\frac{1}{n^2}\sum_{i < j}^n G^n(x_i, x_j)$$
we have
	$$d\gamma_n =
	e^{-n\beta_n W_n }
	d(e^{- \xi V} l)^{\otimes_n}.$$
Now we can try to apply Theorem 
\ref{theorem: laplace}
to get the Laplace principle.
Define
$$G_1(x,y) = G(x,y) +  \epsilon V(x) + 
 \epsilon V(y), \ \ \ 
 W^1_n(x_1,...,x_n) = 
	\frac{1}{n^2}\sum_{i < j}^n G_1(x_i, x_j),$$
$$G_2(x,y) = 
(1-\epsilon)V(x) + (1 - \epsilon)V(y), \ \ \
 W^2_n(x_1,...,x_n) = 
	\frac{1}{n^2}\sum_{i < j}^n G_2(x_i, x_j)$$
and
$$a_n = 
\frac{1}{1-\epsilon}\left( 
\frac{1}{n-1}
	\left(n - \frac{n}{\beta_n} \xi \right) - 
	\epsilon \right)
	\to 1.$$
This definitions allow us to write	
		$$W_n = W^1_n + a_n W^2_n.$$
We start by proving the 
{\it lower limit assumption} 
\ref{lowerlimitassumption}
 and the {\it upper limit assumption} 
 \ref{upperlimitassumption}.
 
{\bf Lower and upper limit assumption, 
\ref{lowerlimitassumption}
and \ref{upperlimitassumption}}. 
By the hypotheses, 
we can see that $G_1$ and $G_2$ are lower
semicontinuous functions bounded from below.		
Then, we can apply Proposition \ref{proposition: k-body} about
the k-body interaction to get that 
$\{W^1_n\}_{n \in \mathbb N}$ 
and $\{W^2_n\}_{n \in \mathbb N}$ 
are {\it stable sequences} 
\ref{stablesequence}
and if we define
the lower semicontinuous functions
$W^1(\mu) = \frac{1}{2} \int_{M \times M}
 G_1(x,y) d \mu(x)  d \mu(y)$
and
$W^2(\mu) = \frac{1}{2} \int_{M \times M}
 G_2(x,y) d \mu(x)  d \mu(y)$,
 then
 $(\{W^1_n\}_{n \in \mathbb N}, W^1)$
 and  $(\{W^2_n\}_{n \in \mathbb N}, W^2)$
 satisfy the {\it lower limit assumption} 
 \ref{lowerlimitassumption}
 and the {\it upper limit assumption} 
 \ref{upperlimitassumption}.
 
Then, as $a_n >0$ for $n$ large enough,
we get that $\{W_n\}_{n \in \mathbb N}$ 
is a {\it stable sequence} 
\ref{stablesequence}
for $n$ large enough.
Noticing that
	$$W^1(\mu) + W^2(\mu) = W(\mu)
	 = \frac{1}{2}\int_{M\times M}
	\left(G(x,y) + V(x) + V(y) \right)  d\mu(x) 
	d\mu(y).$$
we obtain
that $(\{W_n\}_{n \in \mathbb N}, W)$
satisfies the {\it lower limit assumption} 
\ref{lowerlimitassumption}
and the {\it upper limit assumption} 
\ref{upperlimitassumption}.

{\bf Confining sequence \ref{confiningsequence}}.
By Proposition $\ref{proposition: kbodyconfining}$ about
the confining
assumption in the k-body interaction 
and by the 
fact that $G(x,y) + V(x) + V(y) \to \infty$
when $x,y \to \infty$ at the same time, 
we get that
$\{W^1_n + W^2_n\}_{n \in \mathbb N}$ is a
{\it confining sequence} \ref{confiningsequence}.
Along with the fact that 
$\{W_n^1\}_{n \in \mathbb N}$ and
$\{W_n^2\}_{n \in \mathbb N}$ are 
{\it stable sequences}
\ref{stablesequence}
and that $a_n \to 1$ this implies that
$\{W_n\}_{n \in \mathbb N}$
is also a {\it confining sequence} 
\ref{confiningsequence}.

{\bf Regularity assumption 
\ref{regularityassumption}}.
By an argument similar to the one given
in the proof of Theorem \ref{theorem: weakly} 
we can prove
the {\it regularity assumption} 
\ref{regularityassumption} for $W$.

We have proved the conditions to apply Theorem
 \ref{theorem: laplace}.

\end{proof}

\end{theorem}

\subsection{Gaussian random polynomials}

In this subsection we
will see that \cite[Theorem 1]{zelditch}
is a consequence of Corollary \ref{cor: deviations}.
Consider 
a
probability measure $\nu \in \mathcal P(\mathbb C)$
and a continuous function
$\phi: \mathbb C \to \mathbb R$
such that
\begin{equation}
\label{eq: confining}
	\liminf_{z \to \infty}
	 \left\{\phi(z) - 2 \log \|z\|
	\right\} > -\infty.
\end{equation}	
Denote by $\mathbb C_n[z]$ the space of 
complex polynomials
of degree less or equal than $n$ and denote
by $j_n: \mathbb C_n[z]\backslash \mathbb C_{n-1}[z]
 \to \mathcal P(\mathbb C)$
the application that gives the empirical measure
of the zeros of a polynomial, i.e. $j_n$ is defined
by
	$$j_n(p) = \frac{1}{n}\sum_{i=1}^n \delta_{z_i}
	\  \mbox{ if } p(z) = a\prod_{i=1}^n (z-z_i)
	\mbox{ for some } a \neq 0.$$
We shall consider the complex Gaussian measure
$\mathcal G_n$ with covariance
$\langle \cdot , \cdot \rangle_n$ on
$\mathbb C_n[z]$ given by
$$\langle p , q \rangle_n = \int_{\mathbb C}
\bar p(z) q(z) e^{-n \phi(z)} d \nu(z)$$
where we have supposed that
$\langle \cdot, \cdot \rangle_n $ 
is non-degenerate. We will see that
the zeros of a random polynomial
chosen according to $\mathcal G_n$ can be treated
by Corollary \ref{cor: deviations}.
In other words, we are interested
in the pushforward measure of the restriction
of $\mathcal G_n$ to
$\mathbb C_n[z]\backslash \mathbb C_{n-1}[z]$
by $j_n$, that we will denote
by $j_n(\mathcal G_n)$ and that is still
a probability
measure because
$\mathcal G_n( \mathbb C_{n-1}[z]) = 0$, and
we want to write it in the
form (\ref{eq: gibbs}).

\begin{proposition}[Gibbs measure form
of the zeros of a random polynomial]
\label{proposition: gibbspolynomial}
Define the function
$G: \mathbb C \times \mathbb C \to (-\infty,\infty]$
by
	$$G(z,w) = -2\log\|z - w\| + \phi(z) + \phi(w).$$
Then, by the condition (\ref{eq: confining}),
$G$ is a lower semicontinuous
function bounded from below.
Also, by (\ref{eq: confining}), 
$\int_{\mathbb C} e^{-2\phi(z)} 
d \mathcal Leb(z) < \infty$. 
Define
$\pi \in \mathcal P(\mathbb C)$ by
	$$d\pi(z) =\frac{e^{-2 \phi(z)}}
	{ \int_{\mathbb C} e^{-2\phi(z)} 
	d \mathcal Leb(z)}	 
	d\mathcal Leb(z),$$
the symmetric measurable function 
$w_n: \mathbb C^n \to (-\infty,\infty]$
by
\begin{equation}
\label{eq: random polynomial gas}
	w_n(z_1,...,z_n)
	=\frac{1}{n^2} \sum_{i < j} G(z_i,z_j)
	+\frac{n+1}{n^2}
	\log 
	\left( \int_{\mathbb C}
	 e^{-\sum_{i=1}^{n} G(z,x_i)} d\nu(z)
	\right)
\end{equation}
and the Gibbs measure $\gamma_n$
by  
 $$d\gamma_n = e^{-n^2 w_n}d\pi^{\otimes_n}.$$
Then the zeros of a random polynomial chosen
according to $\mathcal G_n$ follows the law
$\frac{\gamma_n}{\gamma_n(\mathbb C^n)}$.
More precisely,
	$$j_n(\mathcal G_n) = 
	i_n \left(
	\frac{\gamma_n}{\gamma_n(\mathbb C^n)}
	\right)$$ 
where $i_n \left(
	\frac{\gamma_n}{\gamma_n(\mathbb C^n)}
	\right)$
denotes the pushforward measure of
$\frac{\gamma_n}{\gamma_n(\mathbb C^n)}$
by $i_n$.	
\begin{proof}
The lower semicontinuity of $G$ 
follows from the continuity of the logarithm 
and the continuity
of $\phi$. As $ -2\log\|z - w \| 
\geq -2\log 2 - 2 \log \|z\| - 2 \log \|w\| $
if $\|z\|, \|w\| \geq 1$ and
using (\ref{eq: confining}) we know that
$G$ is bounded from below. By 
(\ref{eq: confining}) there exists $C > 0$
such that
$e^{-2\phi(z)} \leq C\|z\|^{-4}$ if $\|z\|$
is large enough and we obtain
that 
$\int_{\mathbb C} e^{-2\phi(z)} 
d \mathcal Leb(z) < \infty$.

The statement about  
$j_n(\mathcal G_n)$ is a consequence
of \cite[Theorem 5.1]{raphael}
and the fact that
\begin{align*}
	\prod_{i<j}\|z_i - z_j\|^2 &
	\left(\int_\mathbb C 
	\prod_{i=1}^n \|z - z_i\|^2 
	e^{-n \phi(z)} d\nu(z)\right)^{-(n+1)} 
	d \mathcal Leb^{\otimes_n}(z_1,...,z_n)		
							\\
	&= 
	e^{-\sum_{i<j} G(z_i,z_j)}
	\left(\int_\mathbb C 
	e^{- \sum_{i=1}^n G(z_i,z)}
	 d\nu(z)\right)^{-(n+1)} 
	 \prod_{i=1}^n
	d\pi^{\otimes_n}(z_1,...,z_n)		.
\end{align*}

\end{proof}
\end{proposition}

The energy in (\ref{eq: random polynomial gas})
is a sum of an energy of the 2-body interaction
type and a different kind of energy that
we will try to understand.
Under appropriate conditions in $\phi$,
the authors of \cite{zelditch}
extend $G$ to $\bar {\mathbb C} \times 
\bar {\mathbb C}$ so we shall only consider compact
spaces.

Consider 
$G: M \times M \to (-\infty,\infty]$
a
lower semicontinuous function
on a compact metric space $M$.
Consider $\nu \in \mathcal P(M)$ a probability
measure on $M$
and denote its support by $K \subset M$.
Define $W_n : M^n \to [-\infty, \infty)$ by
$$W_n(x_1,...,x_n) =
\frac{1}{n}\log 
\left( \int_M e^{-\sum_{i=1}^{n} G(z,x_i)} d\nu(z)
\right)$$
and $W: \mathcal P(M) \to [-\infty,\infty)$
by
$$W(\mu) = - \inf_{x \in K}
	\left\{ \int_M G(x,y) d\mu(y) \right\}.$$
Notice that
$\{W_n\}_{n \in \mathbb N}$ is uniformly
bounded from above
and that it is not immediate
to say that $\{W_n\}_{n \in \mathbb N}$
is a {\it stable sequence} \ref{stablesequence}.
	
\begin{lemma}[Upper limit properties]
$W$ is upper semicontinuous and
for each $\mu \in \mathcal P(M)$ we have that
\begin{equation}
\label{eq: upper polynomial}
	\limsup_{n \to \infty} \mathbb E_{\mu^{\otimes_n}}
		 [W_n] \leq W(\mu).
\end{equation}
\begin{proof}

{\bf $W$ is upper semicontinuous.}
This can be seen as a consequence
of the lower semicontinuity of the function
$T: M \times \mathcal P(M) \to (-\infty,\infty]$
defined by $T(x,\mu) = \int_M G(x,y) d\mu(y)$
as follows.
Suppose $\mu_n \to \mu$ in $\mathcal P(M)$
and take $x_n \in K$ such that
$T(x_n,\mu_n) \leq 
\inf_{x \in K} T(x,\mu_n) + \frac{1}{n}$.
Then
$$\liminf_{n\to \infty}
T(x_n,\mu_n) \leq 
\liminf_{n \to \infty}
\left[\inf_{x \in K} T(x,\mu_n) \right].$$
Take a subsequence such that
$\lim_{j \to \infty}
T(x_{n_j}, \mu_{n_j})
=\liminf_{n\to \infty}
T(x_n,\mu_n) $
where, by taking a further subsequence if necessary,
we may assume that $x_n$ converge to some
$x_{\infty} \in K$. The lower semicontinuity
of $T$ implies that
$T(x_{\infty}, \mu)
\leq \lim_{j \to \infty}
T(x_{n_j}, \mu_{n_j})$ and so
	$$\inf_{x \in K} T(x,\mu) \leq
	\liminf_{n \to \infty}
	\left[\inf_{x \in K} T(x,\mu_n) \right].$$
	
{\bf Proof of (\ref{eq: upper polynomial}).} 
Notice that
$$\tilde W_n(\hat \mu) =
\frac{1}{n}\log 
\left( \int_M e^{-n \int_M G(z,x) d\hat \mu(x)} d\nu(z)
\right)$$
if $\hat \mu \in i_n(M^n)$,
where $\tilde W_n$ is defined by
(\ref{eq: extension}).	Then,
if $\hat \mu \in i_n(M^n)$, we have
	$$\tilde 
	W_n(\hat \mu) \leq 
	- \inf_{z \in K} \int_M G(z,x) d\hat \mu(x)
	= W(\hat \mu).$$
Let $\mu \in \mathcal P(M)$, then
	$$\mathbb E_{\mu^{\otimes_n}}[W_n]
	= \mathbb E_{i_n(\mu^{\otimes_n})}[\tilde W_n]
	\leq \mathbb E_{i_n(\mu^{\otimes_n})}[W]$$
and so
	$$\limsup_{n \to \infty}
	\mathbb E_{\mu^{\otimes_n}}[W_n]
	\leq 
	\limsup_{n \to \infty}
	\mathbb E_{i_n(\mu^{\otimes_n})}[W]
	\leq W(\mu)
	$$
by the upper semicontinuity and upper boundedness
of $W$.	

\end{proof}
\end{lemma}
	
We see that the 
{\it upper limit assumption, \ref{upperlimitassumption}},
with the sequence $\{W_n\}_{n \in \mathbb N}$
not necessarily a {\it stable sequence} \ref{stablesequence},
is satisfied in a very general context.
This is not the case for the
{\it lower limit assumption, \ref{lowerlimitassumption}}
and we will state the two main conditions
that allow us to obtain it.

\begin{definition}[Bernstein-Markov condition]
For any $\vec x = (x_1,...,x_n) \in M^n$ consider
the application $s_{\vec x}: M \to \mathbb R$
defined by
$s_{\vec x}(y)=  e^{-\sum_{i=1}^nG(x_i,y)}$
and denote the support of $\nu$ by
$K$.
We say that $(G,\nu)$ satisfies
the Bernstein-Markov condition if the following
is true.
For every $\epsilon>0$ there exists
$C>0$ such that 
$$
\sup_{y \in K} s_{\vec x}(y)
\leq C e^{\epsilon n}\|s_{\vec x}\|_{L^1(M,\nu)}
$$
for every $\vec x \in M^n$ and for every $n > 0$.
\end{definition}

\begin{definition}[Regular pair]
We will say that the pair
$(G,K)$ is regular if the following is true.
For every probability measure $\mu \in \mathcal P(M)$
and every $\epsilon > 0$ there exists 
a probability measure $\nu \in \mathcal P(M)$
such that
$\nu \left(\{x \in K : \int_M G(x,y)d\mu(y) 
\leq \inf_{z \in K} 
\int_M G(z,y)d\mu(y) + \epsilon \} \right) = 1
$
and $x \mapsto \int_M G(x,y)d\nu(y)$ is 
finite and continuous.
\end{definition}

Our Bernstein-Markov condition is an easy consequence
of the Bernstein-Markov condition
in the case of random polynomials
(see \cite[Lemma 9]{zelditch})
and our regular pair
condition is a consequence of the
non-thinness of $K$ (see
the proof of the second part
of \cite[Lemma 26]{zelditch}).

\begin{proposition}[Lower semicontinuity 
and lower boundedness]
\label{proposition: regular implies lower}
Suppose the pair $(G,K)$ is regular.
Then $W$ is lower semicontinuous and bounded
from below.

\begin{proof}
{\bf W is bounded from below.}
The regular pair 
condition implies, in particular, that
there exists a probability measure 
$\nu \in \mathcal P(M)$ supported on $K$ such
that $x \mapsto \int_M G(x,y)d\nu(y)$ is 
 continuous. So,
\begin{align*}
	W(\mu) &\geq - 
 	\int_K \left( \int_M G(x,y)d\mu(y) \right)
 	d\nu(x)										\\
 	&=
 	- \int_M \left( \int_K G(x,y)d\nu(x) \right)
 	d\mu(y)										\\
 	&\geq 
 	- \sup_{y \in M}
 	 \int_K G(x,y)d\nu(x) 
\end{align*} 	 
where we have used Fubini's theorem. 
As  $x \mapsto \int_K G(x,y)d\nu(y) $ is continuous,
it is  bounded from above and we have thus proved
that $W$ is bounded from below.

{\bf W is lower semicontinuous.}
Let $\{\mu_n\}_{n \in \mathbb N}$ be a sequence
of probability measures converging to some
$\mu \in \mathcal P(M)$.
We want to prove that
$\liminf_{n \to \infty} W(\mu_n) \geq W(\mu)$.
For $\epsilon > 0$
the regular pair condition says that there exists
$\nu \in \mathcal P(M)$ supported in $K$
such that
\begin{equation}
\label{eq: ineqregularpair}
	\int_M G(x,y)d\mu(y) 
	\leq \inf_{z \in K} 
	\int_M G(z,y)d\mu(y)  + \epsilon
\end{equation}
for $\nu$-almost every $x$
and $x \mapsto \int_M G(x,y)d\nu(y)$ is 
bounded continuous.
Then integrating
\ref{eq: ineqregularpair}
with respect to $\nu$ we get
	$$\int_{M \times M} G(x,y) d\nu(x) d\mu(y) \leq
	\inf_{z \in K} 
	\int_M G(z,y)d\mu(y) + 
	\epsilon$$
but, as $x \mapsto \int G(x,y)d\nu(y)$ is 
bounded continuous
we have that
	$$\int_M \left(\int_M G(x,y) 
	d\nu(x) \right)d\mu_n(y) 
	\to
	\int_M \left(\int_M G(x,y) 
	d\nu(x) \right)d\mu(y).  $$
As $\nu$ is supported in $K$ we know that
$$	\inf_{z \in K} 
	\int_M G(z,y)d\mu_n(y)
	\leq \int_M \left(\int_M G(x,y) 
	d\mu_n(y) \right)d\nu(x) .$$
Taking the upper limit and using Fubini's theorem
we get
	$$\limsup_{n \to \infty}
	\left[\inf_{z \in K} 
	\int_M G(z,y)d\mu_n(y) \right]
	\leq 
		\inf_{z \in K} 
	\int_M G(z,y)d\mu(y) + 
	\epsilon.
	 $$	
As this is true for every $\epsilon > 0$ we conclude
the proof.
\end{proof}
\end{proposition}

\begin{proposition}[Stability
and lower limit assumption]
\label{proposition: regular and bernstein}
Suppose $(G,K)$ is regular and that
$(G, \nu)$ satisfies the Bernstein-Markov condition.
Then $\{W_n\}_{n \in \mathbb N}$ is a
{\it stable sequence} \ref{stablesequence} and
the pair $(\{W_n\}_{n \in \mathbb N},W)$ satisfies the 
{\it lower limit assumption}, 
\ref{lowerlimitassumption}.

\begin{proof}

If we take the logarithm 
on both sides
of the Bernstein-Markov condition,
we get
$$
-\inf_{y \in K} 
\frac{1}{n}\sum_{i=1}^n
G(x_i,y)
\leq 
\frac{1}{n}\log(C) + \epsilon 
+ W_n(\vec x).$$
Equivalently, we have that
$$W(\mu) 
 \leq 
 \frac{1}{n} \log(C) + \epsilon +
 \tilde W_n(\mu)$$
and, in particular, as $W$ is bounded from below,
we obtain that $\{W_n\}_{n \in \mathbb N}$
is a {\it stable sequence} \ref{stablesequence}. 
If $\mu_n \to \mu$ then
$$W(\mu) \leq \liminf_{n \to \infty} W(\mu_n) 
 \leq 
 \epsilon +
 \liminf_{n \to \infty} \tilde W_n(\mu_n).$$
As this is true for every $\epsilon > 0$
we conclude the proof.	

\end{proof}
\end{proposition}
The following corollary immediately implies 
\cite[Theorem 1]{zelditch}.

\begin{corollary}[Zero temperature macroscopic limit]
Suppose that $(G,K)$ is regular and that
$(G,\nu)$ satisfies the Bernstein-Markov condition.
Suppose also that
for every probability measure
$\mu \in \mathcal P(M)$ such
that $\int_{M \times M} G(x,y) d\mu(x) d\mu(y)
< \infty$
there exists a sequence 
$\{\mu_n\}_{n \in \mathbb N }$ of probability measures
on $M$ such that
$D(\mu_n\| \pi) < \infty$
for every $n \in \mathbb N$ and such that
\begin{equation}
\label{eq: regularitypolynomials}
	\lim_{n 
\to \infty}
\int_{M \times M} G(x,y) d\mu_n(x) d\mu_n(y)
=\int_{M \times M} G(x,y) d\mu(x) d\mu(y).
\end{equation}
Define $w_n:M^n \to (-\infty,\infty]$ and
$w: \mathcal P(M) \to (-\infty, \infty]$ by
$$w_n(z_1,...,z_n)=
\frac{1}{n^2} \sum_{i < j} G(z_i,z_j)
	+\frac{n+1}{n^2}
	\log 
	\left( \int_{\mathbb C}
	 e^{-\sum_{i=1}^{n} G(z,x_i)} d\nu(z)
	\right)$$
and
$$w(\mu)= \frac{1}{2}\int_{M \times M}G(x,y) 
d\mu(x) d\mu(y) - 
\inf_{x \in K}	
\left\{ \int_K G(x,y) d\mu(y) \right\}.$$	
Then
$w$ is the zero temperature macroscopic limit of 
$\{w_n\}_{n \in \mathbb N}$.

\begin{proof}
Using Propositions
\ref{proposition: regular and bernstein}
and \ref{proposition: k-body}
we obtain that $w_n$ is well defined,
$\{w_n\}_{n \in \mathbb N}$ is a 
{\it stable sequence} \ref{stablesequence}
and
that $(\{w_n\}_{n \in \mathbb N},w)$ 
satisfies the {\it lower limit assumption}, 
\ref{lowerlimitassumption}.
The {\it regularity assumption},
\ref{regularityassumption}, is implied by
Proposition \ref{proposition: k-body},
the continuity of $w$ and 
(\ref{eq: regularitypolynomials}).

\end{proof}

\end{corollary}

\section{Fekete points and the zero temperature
deterministic case}
\label{section: fekete}

We begin by a fact which standard proof
can be found in \cite{david}.

\begin{proposition}[Convergence of
the infima]
\label{proposition: fekete}
If $W$ is the positive temperature 
macroscopic limit or the zero temperature
macroscopic limit of
a stable \ref{stablesequence} and 
confining \ref{confiningsequence}
sequence
$\{W_n\}_{n \in \mathbb N}$
then
	$$\inf W_n \to \inf W.$$
\end{proposition}

In particular we get the following consequence.

\begin{theorem}[Deterministic Laplace principle]
\label{theorem: gammaconvergence}
If $W$ is the positive temperature 
macroscopic limit or the zero temperature
macroscopic limit of
a stable \ref{stablesequence} and 
confining \ref{confiningsequence}
sequence
$\{W_n\}_{n \in \mathbb N}$
then for every 
bounded continuous function $f: M \to \mathbb R$ 
	$$\inf \{W_n + f \circ i_n \}
	 \to \inf \{ W + f \}.$$	 	 
\begin{proof}
It is enough to notice that
if 
$W$ is the positive temperature 
macroscopic limit (respectively, the zero temperature
macroscopic limit) of
the sequence $\{W_n\}_{n \in \mathbb N}$
then
$W + f$ is the positive temperature 
macroscopic limit (respectively, the zero temperature
macroscopic limit) of
the sequence 
$\{ W_n + f \}_{ n \in \mathbb N}$ and
use Proposition
\ref{proposition: fekete}.

\end{proof}	 	
\end{theorem}

This may be seen as a natural analogue of 
the Laplace principle. It is just 
(\ref{eq: limit to prove}) without the entropy
term (as if $\beta_n$ were infinity).
This analogue is related to the notion
of $\Gamma$-convergence 
(see \cite{dal} for an introduction
to this topic) as
is said in the following remark.

\begin{remark}[$\Gamma$-convergence]
Theorem \ref{theorem: gammaconvergence}
can be used to prove the 
$\Gamma$-convergence 
of the sequence $\tilde W_n$
defined in (\ref{eq: extension})
(see
\cite[Theorem 9.4]{dal}). In fact,
the confining property of $\{W_n\}_{n \in \mathbb N}$
is not needed as we can obtain
the $\Gamma$-convergence
from
the following standard statement if we take
$A_n$ to be equal to the graph of $\tilde W_n$.

Let $E$ be a measurable space.
Take a sequence $\{A_n\}_{n \in \mathbb N}$
of measurable sets in $E$ and choose
$x \in E$. The following affirmations
are equivalent.

${\bf a)}$ There exists
a sequence $\{X_n\}_{n \in \mathbb N}$ 
of random variables
taking values in $E$
such that 
	$$\forall n \in \mathbb N, \,
	\mathbb P(X_n \in A_n) = 1 
	\ \ \ \mbox{ and } \ \ \ 
	X_n \stackrel{\mathbb P}{\to} x.$$
	
${\bf b)}$ There exists
a sequence $\{x_n\}_{n \in \mathbb N}$ in $E$
such that 
	$$\forall n \in \mathbb N, \, x_n \in A_n
	\ \ \ \mbox{ and } \ \ \ x_n \to x.$$

\end{remark}

{\bf Acknowledgments.}
I would like to thank Raphaël Butez, 
Djalil Chafaï, Adrien Hardy and
the anonymous referees for their useful remarks.

\ \\
\textsc{CEREMADE, UMR CNRS 7534 Université 
Paris-Dauphine, PSL Research university, 
Place du Maréchal
de Lattre de Tassigny 75016 Paris, FRANCE.}\par
  \textit{E-mail address}: \texttt{garciazelada@ceremade.dauphine.fr}

\end{document}